\documentclass[12pt,oneside,pb-diagram]{article}

\newcommand{\Path}{}
\usepackage{ \Path Paper_style_2}
\usepackage{mathtools}
\usepackage{pb-diagram}
\usepackage{string-diagrams}
\usepackage{ \Path Keywords_WB}
\usepackage{ \Path Category}
\usepackage{ \Path Environments}
\usepackage{ \Path Paper_keywords}
\usepackage{\Path tikzcd2}
\usepackage{tikz-cd}
\usepackage{yfonts} 
\DeclareMathOperator{\sps}{\text{\swabfamily SPS}}
\DeclareMathOperator{\graph}{graph}
\newcommand{\join}{\mathfrak{j}}

\newcommand{\CE}[3]{\mathbb{E}_{#3}\left( #1 | #2 \right)}
\newtcbtheorem{Erg}{Interpretation in ergodic theory}{
	enhanced,
	sharp corners,
	attach boxed title to top left={
		yshifttext=-1mm
	},
	colback=ChhaiF,
	colframe=ChhaiB!75!ChhaiC,
	fonttitle=\bfseries\scriptsize,
	fontupper=\scriptsize,
	coltitle=kagoj,
	coltext=kali,
	description font=\bfseries\scriptsize,
	boxed title style={
		sharp corners,
		size=small,
		colback=ChhaiB,
		colframe=ChhaiC!75!black,
		fonttitle=\scriptsize, 
	} 
}{Erg}
\title{Joinings in Markov categories}

\begin{document}
	\author{ Suddhasattwa Das \footnotemark[1], Tomoharu Suda \footnotemark[2] \footnotemark[3]}
	\footnotetext[1]{Department of Mathematics and Statistics, Texas Tech University, USA}
	\footnotetext[2]{Department of Applied Mathematics, Tokyo University of Science, Japan}
	\footnotetext[3]{RIKEN Center for Sustainable Resource Science, RIKEN, Japan}
	\date{\today}
	\maketitle
	\begin{abstract}
		Dynamical systems theory primarily concerns the study of different forms of invariance, such as invariant sets, densities, measures, and observables. Ergodic systems are a special class of dynamical systems which are measure theoretically irreducible. In spite of this specialized property they dominate most discussions on ergodic theory because of the ergodic decomposition theorem. The coexistence and mutual connections between these various coexisting subsystems are illuminated very effectively using the language of Category theory (CT). Some recent advancements have shown how the essence of measure theoretic dynamical systems can be captured through the formalism of Markov categories. This article presents a categorical formulation of the notion of joins. It provides a clear diagrammatic representation of composite and complex dynamical systems is offered by the idea of joins. Several old and new properties of joins are derived solely from their structural properties inherent from their Markovian structure. The idea of a measure preserving dynamical system is first encoded as an object in a comma category. Some classical results such as the relative independence theorem, ergodicity and the pushout theorem, which usually require an analysis of sigma-algebras, are re-established solely from the universal constructs called spans and push-outs. The categorical formulation also reveals a compositional structure inherent in joins, leading to an enriched structure for the category of measure preserving systems.
	\end{abstract}
	
	
	\section{Introduction} \label{sec:intro}
	
	There has been an increased interest \cite[e.g.]{ChoJacobs2019disintegr, fritz2020Stoch, FritzLiang2023gs} in re-axiomatizating probability theory in the graphical language provided by string diagrams of monoidal categories. The motivation is that it allows traditional measure-theoretic probability to be a particular instance, while allowing many of the usual reasoning techniques of probability and statistics without being bound to the set-theoretic details of functions and sigma-algebras. The resulting analysis is known as \emph{categorical probability}, and is mostly done in the setting of Markov categories. Categorical probability is an example of a synthetic theory, as opposed to analytic. The tools of reasoning and definitions are structural properties, rather than properties emerging from pointwise or formulaic definitions.
	
	Parallely a similar effort is underway in dynamical systems theory \cite[e.g.]{mrozek1990leray, RobbinSalamon1992, VagnerSpivakLerman2015, Spivak2015steady, DasSuda2024recon, DasSuda2025enrich} where different aspects of a dynamical system are chosen to be the primitive notion on which every other notion is built. There has been two distinct routes taken to this project. Works such as those based on uncountable measure theory \cite{JamneshanTao2022uncount, Jamneshan2023uncount, JamneshanTao2024uncount} and connecting sets \cite[e.g.]{mrozek1990leray, RobbinSalamon1992} continue to adhere to the classical approach of maps and spaces, but organize them in categories and functors. The other route, such as \cite{Suda2022Poincare, MossPerrone2022ergdc, Suda2023dynamical, DasSuda2025enrich} place the analysis entirely in an abstract category. Concrete notions related to dynamics, such as invariance, orbits and semi-group actions are derived from strictly category theoretic axioms.
	
	This article continues the categorical study of ergodicity initiated in \cite{MossPerrone2022ergdc}. Ergodicity essentially means that the system is irreducible. This irreducibility can be stated equivalently in terms of subsets or measurements. The categorical description replaces sets and measurements with cones and co-cones. The categorical or structural properties that enables these notions to play out, is that of a \emph{Markov category}.  The category $\StochCat$ of Markov kernels is the primary example of a Markov category. Several results of probability theory have recently been reproven in this way, for example, the Kolmogorov and Hewitt–Savage zero-one laws \cite{FritzRischel2020inf} and the de Finetti theorem \cite{FritzEtAl2021definetti}. 
	
	The main feature of Markov categories is that they simultaneously uphold both regular compositions of maps, and convex combinations of probability measures, as the composition rule innate to the category. The most important outcome of this feature is that the fundamental notions of disintegration of measures, and ergodic decomposition are ingrained naturally within the Markov category. Moreover, the usual tools of reasoning in probability, such as almost-sure equality and determinism also emerge as a consequence of the Markov structure. In the instance of $\StochCat$, these notions take their usual meaning from classical probability.
	
	Our focus will be on the idea of joinings \cite[e.g.]{glasner2000entropy, weiss1985strictly, LemanczykEtAl1993gaussian, Rue2006join}  in ergodic theory. joinings are a powerful tool used to study the relationships between different dynamical systems. Joinings are a precise way to quantify how much two dynamical systems overlap or interact. Formally, given two measure-preserving dynamical systems $(X, \Sigma_X, \mu_X, \Phi_X^t)$ and $(Y, \Sigma_Y, \mu_Y, \Phi_Y^t)$ , a joining of these two systems is a probability measure $\lambda$ on the product space $X \times Y$ which is invariant under the product flow $\Phi_X^t \times \Phi_Y^t$, and whose marginal distributions on $(X, \Sigma_X)$ and $(Y, \Sigma_Y)$ are $\mu_X$ and $\mu_Y$ respectively. Joinings are thus the dynamics analog of optimal transport. 
	
	Two dynamical systems may have an abundance of joinings, or just one. The simplest joining is the product measure $\mu_X \times \mu_Y$. This represents the case where the two systems co-exist completely independent of each other in the product space. The greatest insight that the theory of joinings provide is the idea of \emph{disjointness}. Two systems are said to be disjoint if the only joining between them is their product join. This is a unique way of characterizing mutual differences. The difference appears as a lack of overlap. When the two systems $X$ and $Y$ are the same, their joinings are called \emph{self-joins} . The study of self-joins reveal various internal structures of the system, such as the presence of symmetries and mixing properties. 
	
	\paragraph{Outline} We begin by reviewing the formalism of Markov categories in Section \ref{sec:formal}. Next we introduce the notion of state preservation in the context of dynamical systems, in Section \ref{sec:sts_prsrv}. The idea of "joins" is also introduced here. Equipped with these notions, we present in Section \ref{sec:graphical} how the graphical calculus of Markov categories can be used to synthesize some core concepts of ergodic theory. Section \ref{sec:sequence} focuses on the important sub-topic of \emph{disjointness} by introducing the notion of sequence spaces. This allows various new perspectives on the notion of \emph{ergodicity}, in Section \ref{sec:erg}. Section \ref{sec:join_compose} presents a topic of independent interest -- the join diagrams may be composed, which enables Markov categories to be a category enriched over metric spaces. Finally Section \ref{sec:applic} presents some applications of this theory to other Markov categories.
	
	
	\section{Categorical formulation} \label{sec:formal}
	
	We now begin the basic notions of Category theory which would be sufficient to describe measure preserving dynamical systems. Assuming that the readers are familiar with the basic notions from Category theory and the notion of a \emph{monoidal category} \cite[e.g.]{Maclane2013}, we proceed to state our first assumption :
	
	\begin{Assumption} \label{A:monoidal}
		There is a monoidal category $\calX$.
	\end{Assumption}
	
	A category such as $\calX$ plays the role of a mini-universe of spaces and relations or transformations, and provide an ideal setting for studying dynamical systems collectively \cite[e.g.]{MossPerrone2022ergdc, DasSuda2024recon}. Assumption \ref{A:monoidal} is satisfied by some of the most commonly used categories :
	
	\begin{Erg}{}{} 
		The prime example of a monoidal category is $\SetCat$ in which objects are sets up to a pre-fixed cardinality $\kappa$, and morphisms are maps between them. The Cartesian product of sets is a natural monoidal product, which also coincides with their categorical product.
	\end{Erg}
	
	\begin{Erg}{}{} 
		The category $\Topo$ in which objects are topological spaces and morphisms are continuous maps is a monoidal category, with the monoidal product being the usual Cartesian product of topological spaces.
	\end{Erg}
	
	\begin{Erg}{}{} 
		The category $\MeasCat$ in which objects are measurable spaces and morphisms are measurable maps is a monoidal category, with the monoidal product being the usual Cartesian product of topological spaces. Given two measurable spaces $(X, \Sigma_X)$ and $(Y, \Sigma_Y)$ sigma-algebra assigned to $X\times Y$ is the minimal sigma algebra generated by rectangular set $A\times B$ where $A$ is drawn from $\Sigma_X$ and $B$ from $\Sigma_Y$.
	\end{Erg}
	
	At this point we recall the notion of Cartesian categories, which are endowed with a natural monoidal product.
	
	\begin{definition} [Cartesian categories]
		A category equipped with finite categorical products is called a Cartesian category.
	\end{definition}
	
	As mentioned before, the choice of a category places a restriction on the nature of the transition maps and change of variables allowed in the dynamics. For instance, in $\Topo$, $\MeasCat$ and $\VectCat$ respectively, the transitions must be continuous, measurable and linear respectively.
	
	\paragraph{States} The main focus of this article is on states. The Markovian viewpoint allows us to interpret a state as both a particular choice of points from the set of points represented by the space, or as a distribution.
	
	\begin{definition} [States / points]
		Given any monoidal category $\calC$, a \textit{state} or \textit{point} of an object $X$ of $\calC$ is a morphism $x : 1\to X$ from the monoidal unit $1$.
	\end{definition}
	
	\begin{Erg}{}{} 
		Categories such as $\SetCat$, $\MeasCat$ and $\Topo$ are "set-like" or more precisely, \textit{topologically concrete} \cite[see]{DasSuda2024recon}. This means that a state on a space $X$ is just some choice of a point $x$ on it. This is because in such categories, the collection of morphisms $\Hom(1;X)$ is in bijection with the set underlying the object $X$. In $\StochCat$, morphisms $\mu :1 \to \Omega$ are in bijection with probability measures on the probability space $\Omega$.
	\end{Erg}
	
	\begin{definition} [Pointed categories]
		Given any category $\calC$ with a terminal element $1$, its pointed version $\Comma{1}{\calC}$ is the category with in which objects are morphisms such as $x, y$ shown below on the left,
		\[\begin{tikzcd} 1 \arrow[d, "x"'] \arrow[d] \\ X \end{tikzcd}, \,
		\begin{tikzcd} 1 \arrow[d, "y"'] \arrow[d] \\ Y \end{tikzcd} ,\quad
		\begin{tikzcd}
			1 \arrow[d, "x"'] \arrow[d] \arrow[rd, "y"] &   \\
			X \arrow[r, "\phi"]                         & Y
		\end{tikzcd}\]
		and a morphism between them is a morphism $\phi:X\to Y$ which satisfies the commutation shown above on the right. A pointed object such as $x$ above will be denoted as a pair $(X, x)$.
	\end{definition}
	
	\begin{lemma} \label{lem:pcm_cm}
		The pointed version of a monoidal category is also a monoidal category. In the original category is closed, then so is its pointed version.
	\end{lemma}
	
	Lemma \ref{lem:pcm_cm} is proved in Section \ref{sec:proof:pcm_cm}.
	In category theory, when we "point" a category , we are usually looking at the category of objects under the terminal object, denoted $1_{\Context}$. If $\Context$ is already a closed monoidal category (like the category of sets, ), the pointed version $\Comma{1_{\Context}}{\Context}$ inherits a closed monoidal structure. The latter structure is usually a bit different from the original. he natural monoidal product in $\Comma{1}{\Context}$ is the smash \textit{product} $\wedge$. For two pointed objects $(X,_x0)$ and $(Y, y_0)$, the smash product is the pullback $(X,_x0) \wedge (Y, y_0)$
	\begin{equation}
		\begin{tikzcd}
			(X,_x0) \wedge (Y, y_0) \arrow[r, dotted] \arrow[d, dotted] & X \otimes \{y_0\} \arrow[d, hook] \\
			\cup \{x_0\} \otimes Y \arrow[r, hook] & X \otimes Y
		\end{tikzcd}
	\end{equation}

	\paragraph{Dynamical system} A dynamical system can be succinctly described as follows :
	
	\begin{definition} [Dynamical system]
		A dynamical system with context $\calC$ and time structure $\Time$ is a functor $\Phi : \Time \to \calC$. The image under $\Phi$ of the unique object of $\Time$ is called the \emph{domain} of the dynamics.
	\end{definition}
	
	Formulating a dynamical system as a functor from a time semigroup elegantly unifies discrete and continuous dynamics while allowing structural relationships—like conjugacies and factor maps—to be rigorously expressed as natural transformations. This abstraction also seamlessly generalizes to multidimensional group actions and unlocks powerful categorical machinery, such as limits and adjunctions, for analyzing complex systemic behavior across diverse state spaces. Almost all theoretical considerations of a dynamical system involves some notion of invariance or the other. Invariance takes a particularly simple and graphical format in a categorical description.
	
	\begin{definition} [Invariant object] \label{def:invar:1}
		An invariant object of the dynamics $\Phi : \Time \to \calC$ is a cone above the functor $\Phi$. This means an arrow $h:A\to \Omega$ such that the following commutations hold :
		\[
		\begin{tikzcd}
			&& A \arrow[dll, "h"'] \arrow[drr, "h"] \\
			\Omega \arrow[rrrr, "\Phi^t"] && && \Omega
		\end{tikzcd} , \quad \forall t\in \Time.
		\]
		The $\calC$-object $A$ here is called the \emph{apex} of this cone.
	\end{definition}
	
	\begin{definition} [Invariant state] \label{def:invar:2}
		An invariant state of the dynamics $\Phi : \Time \to \calC$ is an invariant object whose apex is the unit object $1$. Equivalently this it is cone above the functor $\Phi$ with apex $1$. This means an arrow $h:1\to \Omega$ such that the following commutations hold :
		\[
		\begin{tikzcd}
			&& 1 \arrow[dll, "h"'] \arrow[drr, "h"] \\
			\Omega \arrow[rrrr, "\Phi^t"] && && \Omega
		\end{tikzcd} , \quad \forall t\in \Time.
		\]
	\end{definition}
	
	\begin{definition} [Invariant measurement] \label{def:invar:3}
		An invariant measurement of the dynamics $\Phi : \Time \to \calC$ is a co-cone below the functor $\Phi$. This means an arrow $h:\Omega\to A$ such that the following commutations hold :
		\[
		\begin{tikzcd}
			\Omega \arrow[drr, "h"'] \arrow[rrrr, "\Phi^t"] && && \Omega \arrow[dll, "h"] \\
			&& A
		\end{tikzcd} , \quad \forall t\in \Time.
		\]
		The $\calC$-object $A$ here is called the \emph{bottom} of this cone.
	\end{definition}
	
	\begin{Erg}{}{} 
		A dynamical system in $\MeasCat$ is simply a choice of a domain $\Omega$, along with a collection of measurable maps $\SetDef{\Phi^t : \Omega \to \Omega}{t\in \Time}$. An invariant object $A$ as in Definition \ref{def:invar:1} is some measurable space $A$ along with a measurable map $h : A\to \Omega$ which remains invariant under the action of $\Phi$. Note that this means that the sub-object of $\Omega$ created by the image $h(A)$ is an invariant, measurable subset of $\Omega$. Invariant states and observations carry a similar interpretation.
	\end{Erg}
	
	\section{State preservation} \label{sec:sts_prsrv} 
	
	Dynamical systems theory is fundamentally about the study of measure-preserving dynamics, capturing how state spaces evolve over time while rigorously conserving their underlying probability structures. By employing strategic changes of variables—such as topological conjugacies or measure-theoretic isomorphisms—complex, non-linear flows can be mapped into simpler, equivalent models. This framework ultimately allows for the classification of long-term systemic behavior by isolating the invariant, ergodic properties that persist across these transformations. However, when these evolving systems involve inherent randomness or uncertainty, traditional models struggle to capture the full picture. Markov categories become absolutely indispensable because they provide a natural, intuitive framework to describe exactly how information is shared, scattered, or permanently lost as unpredictable processes unfold.
	
	\paragraph{Markov categories} While a standard monoidal category elegantly models parallel processes, it lacks the inherent algebraic structure to duplicate variables or discard information. A Markov category introduces explicit "copy" and "discard" morphisms, providing the precise machinery necessary to formally define foundational probabilistic concepts like marginalization, determinism, and conditional independence.
	
	\begin{definition} [Markov category]
		\cite[Def 2.1]{fritz2020Stoch} A symmetric monoidal category $\calC$ is said to be a Markov category if for every object $x\in \calC$, there are the following morphisms
		\[ \copyF_x : x \to x\otimes x , \quad \deleteF_x : x \to 1_{\calC} , \]
		satisfying the following coherence conditions :
	\end{definition}

	\begin{equation} \label{eqn:def:MarkovCat:1}
		\begin{tikzcd}
			& & x \otimes x \otimes x \arrow[d, "\cong"] \\
			x\otimes x \arrow[bend left = 20]{urr}{ \Id_x \otimes \copyF_x } \arrow[bend right = 20]{drr}[swap]{ \copyF_x \otimes \Id_x } & x \arrow[l, "\copyF_x"] \arrow[r, dotted] & x \otimes x \otimes x \\
			& & x \otimes x \otimes x \arrow[u, "\cong"']
		\end{tikzcd} ; \quad 
		\begin{tikzcd}
			x \arrow[dashed, drr, "="] \arrow{d}[swap]{ \copyF_x } \arrow{rr}{ \copyF_x } & & x \otimes x \arrow{d}{ \Id_x \otimes \deleteF_x } \\
			x\otimes x \arrow{rr}[swap]{ \deleteF_x \otimes \Id_x } & & x
		\end{tikzcd} ; \quad 
	\end{equation}
	\begin{equation} \label{eqn:def:MarkovCat:2}
		\begin{tikzcd}
			1_{\calC} & 1_{\calC} \otimes 1_{\calC} \arrow{l}[swap]{ \cong } \\
			x \otimes y \arrow{u}{ \deleteF_{x\otimes y} } \arrow{ru}[swap]{ \deleteF_x \otimes \deleteF_y } 
		\end{tikzcd} ; \quad 
		\begin{tikzcd}
			x \arrow{d}[swap]{ \copyF_{x} } \arrow{r}{ \copyF_x } & x \otimes x \arrow[dl, "\text{swap}_{x,x}"] \\
			x \otimes x
		\end{tikzcd}
	\end{equation}
	and any morphism $f$ satisfies -- 
	\begin{equation} \label{eqn:def:MarkovCat:3}
		\begin{tikzcd}
			x\otimes y \arrow{d}[swap]{ \paran{ \copyF_{x} } \otimes \paran{ \copyF_{y} } } \arrow{dr}{ \copyF_{x\otimes y} } & \\
			x\otimes x \otimes y \otimes y & x\otimes x \otimes y \otimes y \arrow{l}{ \text{symm} }[swap]{\cong}
		\end{tikzcd} ; \quad 
		\begin{tikzcd}
			1_C \\
			x \arrow{u}{ \deleteF_x } \arrow{r}{f} & y \arrow{ul}[swap]{ \deleteF_y }
		\end{tikzcd}
	\end{equation}

	\begin{Erg}{}{} 
		Every category equipped with finite categorical products and a terminal object are also inherently Markov categories. This means that the categories $\SetCat$, $\Topo$ and $\MeasCat$ previously discussed are also Markov categories.
	\end{Erg}
	
	The deletion morphisms can be used to create a natural notion of projections from monoidal products :
	\begin{equation} \label{eqn:def:projection}
		\begin{tikzcd} [column sep = large]
			B& A \otimes B \arrow[dr, "A\otimes \deleteF_B"'] \arrow[dashed, r, Shobuj, "\pi_1"] \arrow[dl, "\deleteF_A\otimes B"] \arrow[dashed, l, Shobuj, "\pi_2"'] & A \\
			1\otimes B \arrow[u, "\cong"] & & A\otimes 1 \arrow[u, "\cong"']
		\end{tikzcd}
	\end{equation}
	One can similarly have coordinate-wise projections for any finite product.

	\begin{Erg}{}{} 
		The most important example of a non-Cartesian Markov category  is $\StochCat$. Here the objects are measurable spaces, and morphisms are Markov kernels. More precisely, given two measurable spaces $(X, \Sigma_X)$ and $(Y, \Sigma_Y)$ a Markov kernel from the former to the latter is a map $f : X\to \Prob(Y; \Sigma_Y)$ where $\Prob(Y; \Sigma_Y)$ is the collection of $\Sigma_Y$-measurable probability measures on $Y$. The map must satisfy 
		\[\forall A\in \Sigma_Y, \; X \mapsto f(\cdot)(A) \mbox{ is } \Sigma_X- \mbox{measurable} .\]
		The composition rule is given as follows
		\[\begin{tikzcd}
			(X, \Sigma_X) \arrow[dr, dashed, Shobuj, "g\circ f"'] \arrow[r, "f"] & (Y, \Sigma_Y) \arrow[d, "g"] \\ 
			& (Z, \Sigma_Z)
		\end{tikzcd} \,:=\,
		\begin{array}{c} (g\circ f)(x)(C) \\ = \\ \int_Y g(y)(C) df(x)(y) \end{array} ,\quad 
		\left\{\begin{array}{c} \forall x\in X\\ \forall C\in \Sigma_Z \end{array}\right.\]
		The composition rule for Markov kernels is thus a convolution of kernels.
	\end{Erg}
	
	\begin{Erg}{}{} 
		Any measurable map $f : (X, \Sigma_X) \to (Y, \Sigma_Y)$ between measurable spaces is also a Markov kernel, assigning to each point $x$ in $X$ the point-measure $\delta_{f(x)}$. Note how the convolution of such point-mass kernels correspond exactly to the usual composition of maps. This indicates that an easy and important instance of Assumption \ref{A:incl} is when $\calX = \MeasCat$, $\calY = \StochCat$ and the functor $\iota$ is just the inclusion.
	\end{Erg}
	
	\begin{Erg}{}{} 
		The notion of a pushforward of a measure is encoded within the composition rules of $\StochCat$. A measure is a morphism $\mu : 1\to X$, and a deterministic measurable map $f:X\to Y$ is a also a morphism in $\StochCat$. Their convolution $f\circ \mu$ can be easily verified to be the pushforward measure
		\[(f_*\mu)(A) := \mu \paran{ f^{-1}(A) } . \]
	\end{Erg}
	
	\begin{Erg}{}{} 
		In any Cartesian category $\calC$ the copy morphism is the \emph{diagonal morphism} which originates naturally from the product structure. Let an object $X$ of $\calC$ be fixed. Then the copy morphism is the unique dotted morphism shown below which enables the commutation :
		\[\begin{tikzcd}
			&& X \arrow[drr, "\Id"] \arrow[dll, "\Id"'] \arrow[dotted, Shobuj] \\
			X && X \times X \arrow[rr, "\proj_2"'] \arrow[ll, "\proj_1"] && X
		\end{tikzcd}\]
		The copy morphism $\copyF_X$ in this case is also denoted as $\langle \Id_X, \Id_X \rangle$.
	\end{Erg}
	
	\paragraph{State preservation} Although the nature of a dynamical system to preserve a given state, or conversely, for a state to be invariant under a given dynamical system, are both aspects of invariance. However we place this in a slightly more elaborate context.
	
	\begin{Assumption} \label{A:incl}
		There is a Markov category $\calY$, and a monoidal functor $\iota : \calX \to \calY$.
	\end{Assumption}
	
	Assumption \ref{A:incl} allows us to start with a dynamical system in the category $\calX$, transfer it to $\calY$ via the functor $\iota$, and then look for invariance in this new category $\calY$. Suppose that the category $\calY$ has a richer structure than $\calX$. This would mean that $\calY$ offers a broader context than $\calX$. Then Assumption \ref{A:incl} provides us the option to consider dynamical systems which are structurally limited to the category $\calX$ but look for invariance in the broader context of $\calY$. Recall the notion of a comma category \cite{DasSuda2024recon, Das2024slice}. We now formalize state-preserving systems using the language of comma categories :
	
	\begin{definition} [Multiple context state preservation] \label{def:sps}
		Given Assumptions  and \ref{A:incl}, the category $\sps$ of state preserving systems is the comma category shown below :
		\[
		\sps := \sps(\iota) := 
		\left[ \begin{tikzcd}[column sep = small]
			\star \arrow[bend right=10]{drr}[swap]{ \Delta_{\calT, \calY}(1) } && && \Functor{\Time}{\calX} \arrow[d, "\subset", hook] \\
			&& \Functor{\Time}{\calY} && \Functor{\Time}{\calX} \arrow[ll, "\iota\circ"]
		\end{tikzcd} \right]
		\]
	\end{definition}
	
	This means that the objects are as follows :
	\[ ob(\sps) = \SetDef{(\mu, \Phi)}{\begin{array}{c} \mbox{Functor } \Phi : \Time \to \calX \,, \\
			\mbox{invariant state } \mu \mbox{ of } \iota\circ \Phi \mbox{ in } \calY 
	\end{array}} \]
	and a morphism $h$ between two objects $(\mu, \Phi)$ and $(\mu', \Phi')$ satisfies the conditions given below :
	\[\begin{tikzcd} (\mu, \Phi) \arrow[d, "h"'] \\  (\mu', \Phi')\end{tikzcd}
	\in \sps \imply
	\begin{tikzcd} \iota \Omega \arrow[d, "h"] \\ \iota \Omega' \end{tikzcd} \,:\,
	\begin{tikzcd}
		1 \arrow[r, "\mu"] \arrow[dr, "\mu'"'] & \iota \Omega \arrow[d, "h"] \\ & \iota \Omega'
	\end{tikzcd} ;\,
	\begin{tikzcd}
		\iota \Omega \arrow[d, "h"'] \arrow[rr, "\iota \circ \Phi^t"] && \iota \Omega \arrow[d, "h"] \\
		\iota \Omega' \arrow[rr, "\iota \circ \Phi'^t"'] && \iota \Omega'
	\end{tikzcd}
	\]
	Observe that in the defining conditions, the morphisms all reside in $\calY$, and involve the images $\iota \Omega$, $\iota \Omega'$ of the $\calX$-objects $\Omega$, $\Omega'$.
	
	\begin{Erg}{}{} 
		The prime instance of Assumptions  and \ref{A:incl} which motivates our categorification is when $\calX=\MeasCat$ and $\calY=\StochCat$. In this case, the category $\sps$ can be represented as a triple $(\Omega, \mu, \Phi)$ where $\Omega$ is a measurable space, $\mu$ a probability measure on $\Omega$ and $\Phi$ is a measurable dynamical system on $\Omega$ which preserves $\mu$. A morphism $h$ as in Definition \ref{def:sps} is a measurable semi-conjugacy $h : \Omega \to \Omega'$ between the dynamics $\Phi$ and $\Phi'$, which also pushes the invariant measure $\mu$ forward into $\mu'$.
	\end{Erg}
	
	Since the category $\sps$ is created as a comma category, it naturally encodes within it the information about the underlying dynamical system, choice of invariant state, and the domain. This is captured in the following commutation diagram :
	\begin{tcolorbox}[colback=ChhaiC, colupper=black, colframe=black, coltext=black, sharp corners, valign=center, top=0mm, bottom=0mm, left=0mm, right=0mm, boxrule=1.5pt]
		\begin{equation} \label{eqn:90dl8}
			\begin{tikzcd}
				\sps \arrow[rr, "\Forget_2"] \arrow[d, "\Forget_3"'] && \Functor{\Time}{\calX} \arrow[dr, "\dom"] \\
				\Comma{1}{\calY} \arrow[rr, "\Forget_2"] && \calY & \calX \arrow[l, "\iota"]
			\end{tikzcd}
		\end{equation}
	\end{tcolorbox}
	The morphisms $\Forget_2$ and $\Forget_3$ emerging from $\sps$ formalize the fact that state-preserving systems have to ingredients -- an dynamical systems object from $\Functor{\Time}{\calX}$ and an arrow $\calY$ from $1$, which is an object of the slice category $\Comma{1}{\calY}$. The commutation in Diagram \eqref{eqn:90dl8} states that the co-domain of this arrow, and the state-space of the dynamics must be the same.
	
	Our definition of the category $\sps$ of state-preserving dynamics involves the following choices in the given order -- firstly a dynamical system $\Phi$ in context-$\calX$ with domain $\Omega$; and secondly a state $\mu : 1 \to \iota \Omega$ which remains invariant under $\Phi$. This definition thus first specifies the dynamics and then a state invariant under the dynamics. One could have alternatively first specified a domain and a state, and then look for dynamics on that domain which preserve that state. We shall see that these two definitions are categorically equivalent. We need to establish certain useful rules of graphical calculus for Markov categories.
	
	\begin{definition} [Specially pointed objects] \label{def:sp_pt_ob}
		The comma category 
		\[ \Comma{1}{\iota} := \left[\begin{tikzcd}[column sep = small]
			\star \arrow[drr, bend right=10, "1"'] && && \calX \arrow[dll, bend left=10, "\iota"] \\
			&& \calY
		\end{tikzcd}\right] \]
		is called the category of \emph{specially pointed objects}. The objects are pairs of the form $(x, X)$ where $X$ is some $\calX$-object and $X : 1\to \iota x$ is a $\calY$ morphism. $X$ is interpreted as a space, and $x$ as a state within $X$ interpreted in the context of $\calY$. A morphism from an object $(x, X)$ to an object $(x', X')$ is some $\calX$-morphism $\phi:X\to X'$ such that $x' = \iota \phi \circ x$.
	\end{definition}
	
	In any category, Markov or not, the objects are interpreted as spaces on domains to host a dynamical system. The idea of pointed categories allows us to treat a pair $(\mu, \Omega)$ jointly as an object or domain that may host a dynamical system. One of the most important realizations for us is the fact that such pointed categories are also Markovian :
	
	\begin{lemma}[Pointed version of Markov category is Markov] \label{lem:ptMarkov_Markv}
		The pointed version of a Markov category is a Markov category. The copy morphism is given by the copied product :
		\begin{equation} \label{eqn:ptMarkov_Markv:1}
			\begin{tikzcd} 1 \arrow[d, "x"'] \\ X \end{tikzcd} 
			\begin{tikzcd} {} \arrow[rr, "\copyF", mapsto] && {} \end{tikzcd} 
			\begin{tikzcd}
				1 \arrow[r, "x"] \arrow[dr, dashed, Shobuj, "\copyF_x"'] & X \arrow[d, "\copyF_X"] \\
				& X \otimes X
			\end{tikzcd}
		\end{equation}
		The deletion morphism is given by
		\begin{equation} \label{eqn:ptMarkov_Markv:2}
			\begin{tikzcd} 1 \arrow[d, "x"'] \\ X \end{tikzcd} 
			\begin{tikzcd} {} \arrow[rr, "\text{delete}", mapsto] && {} \end{tikzcd} 
			\begin{tikzcd}
				1 \arrow[r, "x"] \arrow[dr, dashed, Shobuj, "\cong"'] & X \arrow[d, "\deleteF_X"] \\
				& 1
			\end{tikzcd}
		\end{equation}
		Moreover, this Markov structure is inherited by the slice category $\Comma{1}{\iota}$ of specially pointed objects (Definition \ref{def:sp_pt_ob}).
	\end{lemma}
	
	Since Lemma \ref{lem:pcm_cm} already establishes the monoidal properties of the pointed version, one only needs to verify the rules \eqref{eqn:def:MarkovCat:1}--\eqref{eqn:def:MarkovCat:3} of copy and deletion are  satisfied. Lemma \ref{lem:ptMarkov_Markv} is proved in Section \ref{sec:proof:ptMarkov_Markv}.
	
	\begin{theorem} [Equivalence of cones and state preservation] \label{thm:StPrsrv}
		Suppose that Assumptions \ref{A:monoidal} and \ref{A:incl} hold. Then the category $\sps$ has two equivalent definitions :
		\begin{enumerate} [(i)]
			\item the category $\Functor{\Time}{\Comma{1}{\iota}}$ of dynamical systems on specially pointed objects; 
			\item $\calY$-cones with apex $1$ on dynamical systems in $\calX$, as presented in Definition \ref{def:sps}.
		\end{enumerate}
	\end{theorem}
	Theorem \ref{thm:StPrsrv} is proved in Section \ref{sec:proof:StPrsrv}. We are now finally able to define a join in the context of state-preserving dynamics. 
	
	\paragraph{Joins} A join is basically a diagram commonly known as a \emph{span} \cite{dawson2004span, dawson2004universal, heindel2011being} in category theory. Given an ordered pair of objects $x,y$ of a category $\calC$, a span is any diagram of the form
	\[\begin{tikzcd}
		& z \arrow[dr] \arrow[dl] \\
		x & & z
	\end{tikzcd}\]
	involving two morphisms to $x,y$ from a common third object $z$. When the category is $\sps$, such a diagram can be meaningfully expressed as shown on the left below :
	\[\begin{tikzcd} [scale cd = 0.9, column sep = small]
		& (\Omega, \mu, \Phi) \arrow[dr, "h_2"] \arrow[dl, "h_1"'] \\
		(\Omega_1, \mu_1, \Phi_1) & & (\Omega_2, \mu_2, \Phi_2)
	\end{tikzcd}
	\begin{tikzcd}{} \arrow[rr, mapsto, " \Forget_2"] && {}\end{tikzcd}
	\begin{tikzcd}
		& (\Omega, \Phi) \arrow[dr, "\Forget_2(h_2)"] \arrow[dl, "\Forget_2(h_1)"'] \\
		(\Omega_1, \Phi_1) & & (\Omega_2, \Phi_2)
	\end{tikzcd}\]
	The diagram on the left exists in $\calY$. Applying the forgetful functor $\Forget_2$ from \eqref{eqn:90dl8} we get the $\calX$-diagram on the right. This span will be called a {join} of the two state-preserving systems $(\mu_1, \Phi_1)$ and $(\mu_2, \Phi_2)$ if the apex of the diagram on the right is the tensor product of the two constituent systems. In other words the following two conditions are satisfied.
	\begin{equation} \label{eqn:def:ddp8z}
		\begin{tikzcd} [scale cd = 0.9, column sep = small]
			& (\Omega, \mu, \Phi) \arrow[dr, "h_2"] \arrow[dl, "h_1"'] \\
			(\Omega_1, \mu_1, \Phi_1) & & (\Omega_2, \mu_2, \Phi_2)
		\end{tikzcd} ;\;
		\begin{tikzcd} [scale cd = 0.9]
			& (\iota \Omega, \iota \Phi) \arrow[dd, "\cong"] \arrow[dr, "\Forget_2(h_2)"] \arrow[dl, "\Forget_2(h_1)"'] \\
			(\iota \Omega_1, \iota \Phi_1) & & (\iota \Omega_2, \iota \Phi_2) \\
			& \paran{ \iota \Omega_1 \otimes \iota \Omega_2, \iota \Phi_1 \otimes \iota \Phi_2 } \arrow[ur, "\pi_2"'] \arrow[ul, "\pi_1"]
		\end{tikzcd}
	\end{equation}

	\begin{Erg}{}{} 
		In ergodic theory, a joinings is simply a probability measures on the product space $\Omega_1 \otimes \Omega_2$ which is invariant under the product dynamics and whose two marginals are $\mu_1$ and $\mu_2$ respectively. Ergodic theory is about the spectral or functional properties of invariant measures, and the collection of joinings between two measure preserving systems encodes both the independence and structural alignment between them. 
	\end{Erg}
	
	Join-theory is essentially the analysis of the space of invariant couplings between two measure-preserving transformations. Any two $\sps$-objects have a trivial joining, which is the monoidal product of the two states. In the context of ergodic theory, they play the role of independent coupling of two dynamical systems. We define
	
	\begin{definition} [Disjointness]
		\cite{Furstenberg1967} Two $\sps$ objects are said to be disjoint if their only join is their product join. More generally, given a v-diagram as in \eqref{eqn:vee_diag}, the $\sps$-objects $(\Omega_1, \mu_1, \Phi_1)$ and $(\Omega_2, \mu_2, \Phi_2)$ are said to be disjoint over the common factor $(\Omega_3, \mu_3, \Phi_3)$ if the only element of $J(\mu_1, \mu_2 \mid h_1, h_2 \mid \mu_3)$ is $\join(\mu_1, \mu_2 \mid h_1, h_2 )$
	\end{definition} 
	
	Thus two systems are disjoint if their only joining is the independent product measure, meaning that they share no non-trivial factors or structural similarities. The notion of \textit{disjointness} was championed by Furstenberg and others \cite[e.g.]{Furstenberg1967, Furstenberg1977ergodic} brings to light various fundamental properties to be a consequence of disjointness.  Beyond establishing spectral and factor-independent boundaries, joinings provide the foundational language for the Furstenberg–Zimmer structure theory which provides a universal decomposition of systems into rigid and weakly mixing components. Of more interest are the non-trivial joinings that may exist. We new review a few diagrammatic means of constructing joinings. The techniques reveal the power of graphical calculus provided by monoidal categories. 
	
	\section{Graphical calculus with joinings} \label{sec:graphical}
	
	The graphical calculus of Markov categories employs intuitive string diagrams—where wires represent spaces and nodes denote probabilistic processes—to visually encode the flow, duplication, and discarding of information. This rigorous topological language transforms complex algebraic equations into straightforward diagrammatic manipulations, vastly simplifying proofs involving conditional independence, marginalization, and the behavior of composite stochastic systems. The first simple instance is 
	
	\begin{definition} [Graphs]
		Given any morphism $f$ as shown below on the left :
		\begin{equation}
			\begin{tikzcd}x \arrow[d, "f"'] \\ y \end{tikzcd} \imply 
			\begin{tikzcd}
				x \arrow[dr, Shobuj, dashed, "\graph(f)"'] \arrow[rrr, "\copyF_x"] &&& x\otimes x \arrow[dll, "x\otimes f"] \\
				& x\otimes y
			\end{tikzcd}
		\end{equation}
		its graph is the composite morphism shown above on the right, created with the aid of the $\copyF$ morphism.
	\end{definition}

	\begin{definition} [Graph-joinings] \label{def:graph_join}
		\cite[Eqn 1]{LemanczykEtAl1993gaussian} Given a morphism $h : (\Omega_1, \mu_1, \Phi_1) \to (\Omega_2, \mu_2, \Phi_2)$ one can construct a state $\lambda_h$ which is a composite of the two horizontal morphisms shown below :
		\begin{equation} \label{eqn:def:graph_join:1}
			\begin{tikzcd} [column sep = large]
				& && \iota \Omega_2\\
				1 \arrow[urrr, bend left=20, "h\mu_1 = \mu_2", dashed] \arrow[drrr, bend right=20, "\mu_1"', dashed] \arrow[r, "\mu_1"] & \iota \Omega_1  \arrow[bend left=5, urr, "h", dashed] \arrow[drr, bend right=5, "="', dashed] \arrow[rr, "\graph(h)"] && \iota \Omega_1 \otimes \iota \Omega_2 \arrow[d, "\pi_1"] \arrow[u, "\pi_2"'] \\
				& && \iota \Omega_1
			\end{tikzcd}
		\end{equation}
	\end{definition}
	
	Diagrammatically, this join $\lambda_h$ is given by
	\begin{center}
		\begin{tikzpicture}
			\node[box=1/0/1/0] (n) at (1,-1) {1_X};
			\node[box=1/0/1/0] (h) at (3,-1) {h};
			\node[dot] (cpy) at (2,-2) {};
			\node[box=1/0/1/0] (mu) at (2,-3) {\mu};
			\node[dot] (i) at (2,-4) {};
			\wires{
				n = {south = cpy.north},
				h = {south = cpy.north},
				cpy = {south = mu.north},
				mu= {south = i.north},
			}{n.north, h.north}
		\end{tikzpicture}
	\end{center} 
	
	A graph joining is an example of the application of graphical calculus, in which a simple diagram, namely a morphism $h$ in $\sps$ is used to create a more intricate diagram in $\calY$. To establish this diagram to be a joining, we need a further property called \emph{determinism}.
	
	\paragraph{Determinism} In a Markov category, a morphism $f: X \to Y$ (which generally represents a stochastic or probabilistic process) is interpreted as deterministic if it perfectly preserves the "copy" operation. If $f$ contains any randomness (like flipping a coin), running it independently on two identical inputs will generally yield different results. Therefore, the only way the two sides can be equal is if the process $f$ has no internal randomness at all—meaning it evaluates to the exact same output every single time. Formally this means :
	
	\begin{definition} [Deterministic]
		A morphism $x\xrightarrow{f} y$ is said to be deterministic if the following commutations hold :
		\begin{equation} \label{eqn:def:deterministic}
			\begin{tikzcd}
				x \arrow{d}[swap]{f} \arrow{rr}{ \copyF_x } & & x\otimes x \arrow{d}{ f\otimes f } \\
				y \arrow{rr}{ \copyF_y } & & y\otimes y
			\end{tikzcd}
		\end{equation}
	\end{definition}
	
	This essentially means that the transition under $f$ after copying, is equivalent to copying after the transition of $f$. We therefore assume : 
	
	\begin{Assumption} [Determinism] \label{A:det_base}
		The images of morphisms under $\iota$ from Assumption \ref{A:incl} are deterministic in $\calY$.
	\end{Assumption}
	
	\begin{Erg}{}{}
		In $\StochCat$ where determinism carries its usual meaning, the diagram \eqref{eqn:def:deterministic} asserts that evaluating a deterministic process once and duplicating its output yields the exact same outcome as duplicating the input first and applying the process twice independently on both copies. Stochastic morphisms, i.e. proper Markov kernels fail to commute with copy because running a probabilistic process twice produces independent random outcomes rather than identical duplicates, deterministic maps preserve information perfectly and respect non-linear copying. 
	\end{Erg}
	
	\begin{lemma} \label{lem:graph_join:1}
		Suppose Assumptions \ref{A:monoidal}, \ref{A:incl} and \ref{A:det_base} hold. Given an $\sps$-morphism $h : (\Omega_1, \mu_1, \Phi_1) \to (\Omega_2, \mu_2, \Phi_2)$, the state $\graph(h) \mu_1$ constructed in \eqref{eqn:def:graph_join:1} above creates a joint state for $\mu_1, \mu_2$.
	\end{lemma}
	
	Lemma \ref{lem:graph_join:1} is proved in Section \ref{sec:proof:graph_join:1}. 
	
	\begin{Erg}{}{} 
		In ergodic theory, the graph joining induced by a semiconjugacy (or factor map) $h: (\Omega_1, \mu_1, \Phi_1^t) \to (\Omega_2, \mu_2, \Phi_2^t)$ is the unique invariant probability measure $\mu_h$ on the product space $\Omega_1 \times \Omega_2$ supported entirely on the graph of $h$, explicitly defined by $\mu_h(A \times B) = \mu_1(A \cap h^{-1}(B))$. By concentrating all probability mass on the subset $\{(x, h(x)) \mid x \in \Omega_1\}$, the graph joining acts as a maximally rigid, deterministic coupling where the trajectory in the factor space is entirely subordinate to the state of the extension—a stark contrast to the independent product measure that defines disjointness. Analytically, these deterministic joinings manifest as extreme points in the weakly compact, convex simplex of all joinings between the two systems, serving as the foundational building blocks for defining relatively independent joinings over common factors.
	\end{Erg}
	
	A special case of graph joining are obtained when the morphism $\phi$ shown above is just an iteration of the flow $\Phi^T$ :
	
	\begin{definition} [Off-diagonal joinings] \label{def:off_diag_join}
		\cite[Sec 6.1]{glasner2015join} The join obtained by replacing $(\Omega_1, \mu_1, \Phi_1)$, $(\Omega_2, \mu_2, \Phi_2)$ with $(\Omega, \mu, \Phi)$ in \eqref{eqn:def:graph_join:1} and $h$ with $\Phi^T$ and :
		\begin{equation} \label{eqn:def:off_diag_join}
			\begin{tikzcd} [column sep = large]
				& && \iota \Omega\\
				1 \arrow[urrr, bend left=20, "\Phi^T\mu = \mu", dashed] \arrow[drrr, bend right=20, "\mu_1"', dashed] \arrow[r, "\mu"] & \iota \Omega  \arrow[bend left=5, urr, "\Phi^T", dashed] \arrow[drr, bend right=5, "="', dashed] \arrow[rr, "\graph(\Phi^T)"] && \iota \Omega \otimes \iota \Omega \arrow[d, "\pi_1"] \arrow[u, "\pi_2"'] \\
				& && \iota \Omega_1
			\end{tikzcd}
		\end{equation}
		is called the \textit{off-diagonal joining} at time $T$.
	\end{definition}
	
	Off-diagonal joinings are obviously self-joinings.

	\paragraph{Disintegration of measures} In a Markov category, disintegration is the formal, categorical abstraction of conditional probability. It describes how a joint process can be factored into a marginal process and a conditional process. The notion of disintegration to be used is an abstraction of the rule $P(X,Y) = P(Y|X) P(X)$.
	
	\begin{definition} [Disintegration of a morphism]
		Suppose there is a composable sequence of morphisms $p : A \to X$ and $f : X \to Y$. Then a disintegration of $f$ along $p$ is a factorization of the morphism $\graph(f) \circ p$ via $f \circ p$. More precisely it is a morphism $s : Y \otimes A \to X$ which satisfies the commutation shown below on the right.
		%
		\begin{equation} \label{eqn:def:disinteg:1}
			\forall \begin{tikzcd} A \arrow[d, "p"'] \\ X  \arrow[d, "f"'] \\ Y\end{tikzcd} \;:\; \exists
			\begin{tikzcd} A \otimes Y \arrow[d, Shobuj, "s"'] \\ X \end{tikzcd}
			\;:\;
			\begin{tikzcd} [column sep = small]
				A \arrow[d, "\graph(f \circ p)"'] \arrow[r, "p"] & X \arrow[rr, "\graph(f)"] & & X \otimes Y \\
				A\otimes Y \arrow[rr, "A \otimes \copyF_Y"'] && A \otimes Y \otimes Y \arrow[ur, "s \otimes Y"']
			\end{tikzcd}
		\end{equation}
	\end{definition}
	
	When the object $A=1$, so that $p$ is just a state in $X$, \eqref{eqn:def:disinteg:1} has a more specific and significant meaning : 
	
	\begin{definition} [Disintegration of a morphism along a state]
		Suppose there is a composable sequence of morphisms $p : 1 \to X$ and $f : X \to Y$. Then a disintegration of $f$ along $p$ is a factorization of the morphism $\graph(f) \circ p$ via $\graph(f \circ p)$. More precisely it is a morphism $(f^{-1}|p) : Y\to X$ which satisfies the commutation shown below on the right.
		\begin{equation} \label{eqn:def:disinteg:2}
			\forall \begin{tikzcd} 1 \arrow[d, "p"'] \\ X  \arrow[d, "f"'] \\ Y\end{tikzcd} \;:\; \exists
			\begin{tikzcd} Y \arrow[Shobuj, d, "(f^{-1}|p)"] \\ X \end{tikzcd}
			\;:\;
			\begin{tikzcd} 
				& X \arrow[rr, "\graph(f)"] & & X \otimes Y  \\
				1 \arrow[dr, "f\circ p"'] \arrow[ur, "p"] \\
				& Y \arrow[uurr, dashed, "\graph^T(f^{-1}|p)"] \arrow[rr, "\copyF_Y"'] && Y\otimes Y \arrow[uu, "(f^{-1}|p) \otimes Y"']
			\end{tikzcd}
		\end{equation}
	\end{definition}
	
	Thus, a disintegration of a morphism $f$ along a state $p$ of its domain means the following isomorphism :
	\begin{equation} \label{eqn:def:disinteg:3}
		\SqBrack{ \graph^T(f^{-1}|p) \circ f } \circ p = \SqBrack{ \graph(f) } \circ p .
	\end{equation}
	Thus from the perspective of $p$, $\graph^T(f^{-1}|p) \circ f$ is indistinguishable from $\graph(f)$. In fact there is an inversion lurking within this relation. We now make an important interpretation of disintegration using the language of \textit{almost-sure} equality.
	
	\paragraph{Almost sure equality}  This essential notion was well adapted within the framework of Markov categories by Cho and Jacobs \cite[Def 5.1]{ChoJacobs2019disintegr}. In a Markov category, the concept of "almost everywhere" (or "almost surely") is always defined relative to a specific prior state (a probability distribution) and is formalized by comparing joint distributions. Because Markov categories generally lack underlying sets, one does not have the notion of "sets of measure zero" to be able to define "almost everywhere". Instead, the notion of "almost everywhere" has to interpreted based on how two processes appear to a given prior.
	
	\begin{definition} [Almost sure equality]
		Suppose $\calC$ is a Markov category. Given a morphism $p : \Theta \to X$, two morphisms $f,g : X \to Y$ are said to be $p$-almost surely equal if 
		\[ \graph(f) \circ p = \graph(g) \circ p. \]
	\end{definition}
	
	Thus the morphisms may not be equal, but when considered jointly with their domain, they are equalized by the state $p$. This can be restated through the following commutation diagram :
	\begin{equation} \label{eqn:def:almost_sure}
		\begin{array}{c} \graph(f) \circ p  \\ = \\ \graph(g) \circ p \end{array}
		\;\Leftrightarrow\;
		\begin{tikzcd}
			& X\otimes Y \\
			X\otimes Y \arrow[dr, "\cong"'] \arrow[ur, "\cong"] & \Theta \arrow[l, dotted] \arrow[r, "p"] & X \arrow[ul, bend right=5, "\graph(f)"'] \arrow[dl, bend left=5, "\graph(g)"] \arrow[rr, "\copyF_X"] && X \otimes X \arrow[ulll, bend right=15, "X\otimes f"'] \arrow[dlll, bend left=15, "X\otimes g"] \\
			& X\otimes Y
		\end{tikzcd}
	\end{equation}

	\begin{Erg}{}{} 
		In the context of $\StochCat$, diagram \eqref{eqn:def:almost_sure} implies that for any measurable sets $A\subset X$ and $B\subset Y$, and every point $\theta\in \Theta$,
		\[ p\paran{ \theta, A \cap f^{-1}(B) } = p\paran{ \theta, A \cap g^{-1}(B) } . \]
	\end{Erg}
	
	\begin{lemma} [Disintegration as inversion] \label{lem:disinteg_invrsn}
		Suppose Assumptions \ref{A:monoidal}, \ref{A:incl} and \ref{A:det_base} holds, then disintegration along a deterministic map is inversion almost surely. More precisely let $p : 1\to X$ be a state and $f:X\to Y$ is deterministic. Then the disintegration $(f^{-1}|p) : Y\to X$ is the unique morphism such that
		\begin{equation} \label{eqn:disinteg_invrsn}
			f\circ ( f^{-1}|p ) = \Id_Y, \quad fp-\mbox{almost surely} .
		\end{equation}
	\end{lemma}
	
	The disintegration of a morphism—the categorical equivalent of conditional probability—perfectly showcases the elegance of graphical calculus by translating dense analytical equations into simple diagrammatic factoring. Visually, a joint stochastic process is cleanly split into a marginal state wire that feeds directly into a subsequent conditional channel. This framework transforms the abstract, algebraic decomposition of complex dependencies into an intuitive topological equality, where rigorous proofs are achieved simply by redrawing the flow of information.
	
	\begin{Erg}{}{} 
		In the context of $\StochCat$, disintegration corresponds to the usual notion of disintegration of measures along the fibres of a measurable transformation. Given a measure $p$ and a measurable transformation $f:X\to Y$, for $f_*p$-almost every $y\in Y$, there is a probability measure $p_y$ supported on the inverse image $f^{-1}(y)$ defined uniquely by this weak characterization :
		\[ \int_{x\in X} \phi dp = \int_{y\in Y} \int_{x \in f^{-1}(y)} \phi(x) dp_y(x) , \quad \forall \phi \in L^\infty(p) . \]
		The correspondence 
		\[ y \mapsto \delta_{y} \otimes p_y \]
		is precisely the Markov kernel $(f^{-1}|p) : Y \to X\otimes Y$.
	\end{Erg}
	
	\paragraph{Conditional expectation} Conditional expectation of a random variable $X$ with respect to a sub-sigma algebra $\Sigma'$ is usually interpreted analytically as the optimal orthogonal projection of $X$ onto the restricted subspace of known information created by $\Sigma'$. From a probabilistic or statistical perspective, it is interpreted as the localized update of a distribution from given observed outcomes. In the graphical calculus of Markov categories, both perspectives unify seamlessly under the concept of Bayesian inversion or disintegration. The re-definition requires immense use of the commutation \eqref{eqn:def:disinteg:2}, so we assume
	
	\begin{Assumption} \label{A:disinteg}
		The category $\calY$ from Assumptions \ref{A:incl} allows disintegrations of morphisms along states.
	\end{Assumption}
	
	\begin{lemma}[Conditional expectation] \label{lem:subsig:1}
		Suppose Assumptions \ref{A:monoidal}, \ref{A:incl} and \ref{A:disinteg} hold. Consider a morphism $h$ in $\Comma{1}{\iota}$ as shown below on the left : 
		\begin{equation} \label{eqn:lem:subsig:1}
			\begin{tikzcd}\paran{ \Omega_X, \mu_X } \arrow[d, "h"'] \\ \paran{ \Omega_Y, \mu_Y }     \end{tikzcd}
			\;\Rightarrow\;
			\begin{tikzcd} [column sep = small]
				1 \arrow[dr, "\mu_X"] \arrow[ddr, bend right=10, "\lambda"'] \\ 
				& \Omega_X \arrow[d, "\graph(h)"] \\ 
				& \Omega_X \otimes \Omega_Y
			\end{tikzcd}
			\;\Rightarrow\;
			\begin{tikzcd} 
				\Omega_Y \arrow[d, Shobuj, dotted, "\CE{X}{Y}{\lambda}"'] \arrow{dr}{ (\pi_Y^{-1}|\lambda) }
				\\\Omega_X & \Omega_X \otimes \Omega_Y \arrow[l, "\pi_X"]
			\end{tikzcd}
		\end{equation}
		Let $\lambda$ be the resulting  graph join, as shown above in the middle, leading to the $\calY$-morphism $\CE{X}{Y}{\lambda} : \Omega_X \to \Omega_Y$. shown above as a dotted green arrow. This morphism $\CE{X}{Y}{\lambda}$ creates the following commutation which holds $\lambda$.a.e. :
		\begin{equation} \label{eqn:lem:subsig:2}
			\lambda-\mbox{a.e.} \,:\,
			\begin{tikzcd} [column sep = small]
				& \Omega_X \otimes \Omega_Y \arrow[dl, "\pi_X"'] \arrow[dr, "\pi_Y"] \\
				\Omega_X && \Omega_Y \arrow[ll, dotted, Shobuj, "\CE{X}{Y}{\lambda}"']
			\end{tikzcd}
		\end{equation}
	\end{lemma}
	
	Lemma \ref{lem:subsig:1} is proved in Section \ref{sec:proof:subsig:1}. Note that in site of the direction of $\CE{X}{Y}{\lambda}$ it is to be interpreted as a the conditional expectation of the marginal $X$ w.r.t. $Y$, from the perspective of $\lambda$. This is because one can pre-compose with this morphism to convert morphisms with domain $X$ into morphisms with domain $Y$. This simple mapping can be interpreted as the categorical analog of conditional expectation.
	
	\begin{Erg}{}{} 
		Lemma \ref{lem:subsig:1} has a particularly important implication in the context of $\StochCat$, called an \emph{intertwining} Markov operator \cite[Sec 1.3]{LemanczykEtAl1993gaussian}. The morphism $\CE{X}{Y}{\lambda}$ from \eqref{eqn:lem:subsig:1} is the Markov kernel given by
		\[ \CE{X}{Y}{\lambda} : \Omega_Y \to \Prob(\Omega_X), \quad \CE{X}{Y}{\lambda}(y) := \CE{X}{Y}{\lambda}|_{y}. \]
		This results in an operator $\Phi_{Y\to X}^{\lambda} : L^2(\mu_X) \to L^2(\mu_Y)$
		\[ \paran{ \Phi_{Y\to X}^{\lambda} \phi }(y) := \int_{x\in \Omega_X} \phi(x) d\gamma(y)(x) = \int_{\Omega_X} \phi(x) d (\lambda|_y)(x), \]
		which is also a \emph{Markov} operator, meaning it preserves positivity and maps $1$ to $1$. This Markov operator may be defined equivalently by its dual action :
		\[\begin{split}
			\bracketBig{ \phi \otimes 1_{\Omega_X} , 1_{\Omega_Y} \otimes \psi }_{L^2(\lambda)} &= \int_{\Omega_X \times \Omega_Y} \phi(x) \overline{ \psi(y) } d\lambda(x,y) = \int_{\Omega_Y} \SqBrack{ \int_{\Omega_X} \phi(x) d (\lambda|_y)(x) } \overline{ \psi(y) } d\mu_Y(y) \\
			& = \int_{\Omega_Y} \paran{ \Phi_{Y\to X}^{\lambda} \phi }(y) \overline{ \psi(y) } d\mu_Y(y) \\
			&= \bracketBig{ \Phi_{Y\to X}^{\lambda} \phi, \psi }_{L^2(\mu_Y)}.
		\end{split}\]
	\end{Erg}
	
	\paragraph{Non-trivial joins} Since the notion of a conditional expectation has been adapted within the framework of Markov categories, one obtains an interesting interpretation of non-trivial joinings in terms of the morphism / operator $\CE{X}{Y}{\lambda}$ :
	
	\begin{lemma} [Non-trivial joinings] \label{lem:dij03}
		Consider the following diagram
		\begin{equation} \label{eqn:dij03}
			\begin{tikzcd}
				& 1 \arrow[d, "\lambda"] &  \\
				& X\otimes Y \arrow[ld, "\pi_X"'] \arrow[rd, "\pi_Y"] &  \\
				X & & Y \arrow[ll, "\CE{X}{Y}{\lambda}"']
			\end{tikzcd}
			\imply
			\begin{tikzcd}
				& & & 1 \arrow[d, "\lambda"] &  \\
				& & & X\otimes Y \arrow[ld, "\pi_X"'] \arrow[rd, "\pi_Y"] &  \\
				Z && X \arrow[ll, "\psi_{X,Y,\lambda}"'] &  & Y \arrow[ll, "\CE{X}{Y}{\lambda}"'] 
			\end{tikzcd}
		\end{equation}
		and the morphism $\CE{X}{Y}{\lambda}$ guaranteed by Lemma \ref{lem:subsig:1}.
		Then the state $\lambda$ is non-trivial, i.e., is not the monoidal product of its $X$ and $Y$ marginals iff the morphism $\CE{X}{Y}{\lambda}$ does not factor through $1$. Moreover, there exists a unique (up to isomorphism) deterministic morphism $\psi_{X,Y,\lambda} : X \to Z$ as shown on the right, such that for any morphism $\psi' : Y \to Z'$  such that $\psi' \circ \CE{X}{Y}{\lambda}$ is non-constant, $\psi'$ must factor through $\psi$.
	\end{lemma}
	
	Lemma \ref{lem:dij03} will be a key tool for proving the existence of non-trivial factors. The following observation will be of use to us
	
	\begin{lemma} \label{lem:2fi03}
		The disintegration $(\pi_Y^{-1}|\lambda)$ is the copied product of the conditional  expectation $\CE{X}{Y}{\lambda}$ \eqref{eqn:dij03} and the identity on $Y$. In other words, the following commutation holds :
		\begin{equation} \label{eqn:disinteg_cndtnl}
			\begin{tikzcd}
				Y \arrow[d, "\copyF"'] \arrow[rr, "(\pi_Y^{-1}|\lambda)"] && X \otimes Y \\
				Y \otimes Y \arrow{urr}[swap]{\CE{X}{Y}{\lambda} \otimes Y}
			\end{tikzcd}
		\end{equation}
	\end{lemma}
	
	\begin{Erg}{}{} 
		In the context of $\MeasCat$, Lemma \ref{lem:dij03} provides an intuitive interpretation of non-disjointness. If $\lambda$ is a non-trivial join for probability spaces $(X, \mu)$ and $(Y, \nu)$, then Lemma \ref{lem:dij03} is equivalent to saying that there is a $\phi \in L^2(Y,\nu)$ such that the conditional expectation $\mathbb{E}_{\lambda} \paran{g \mid \Sigma_X \times Y }$ is non-constant.
	\end{Erg}

	\begin{lemma} \label{lem:sps_mrphsm_invert}
		Suppose Assumptions \ref{A:incl} and \ref{A:disinteg} hold. Then a morphism between state-reserving dynamics can be inverted almost everywhere . More precisely, given an $\sps$ morphism as shown on the left :
		\begin{equation} \label{eqn:sps_mrphsm_invert}
			\begin{tikzcd} (\Omega_X, \mu_X, \Phi_X) \arrow[d, "h"'] \\ (\Omega_Y, \mu_Y, \Phi_Y) \end{tikzcd}
			\imply
			\begin{tikzcd} \iota \Phi_Y \arrow[d, Rightarrow, "(h^{-1}|\mu_X)"'] \\ \iota \Phi_X \end{tikzcd}
			\;\Leftrightarrow\;
			\forall t \,:\,
			\mu_Y-\mbox{a.s.} \,:\,
			\begin{tikzcd}
				\iota \Phi_Y \arrow[d, "(h^{-1}|\mu_X)"'] \arrow[r, "\iota \Phi_Y^t"] & \iota \Phi_Y \arrow[d, "(h^{-1}|\mu_X)"] \\
				\iota \Phi_X \arrow[r, "\iota \Phi_X^t"'] & \iota \Phi_X
			\end{tikzcd}
		\end{equation}
		the disintegration of $h$ along $\mu_Y$ serves as a natural transformation between the two dynamics. Equivalently, it implies the commutations $\mu_Y$-almost everywhere as shown on the right.
	\end{lemma}
	
	Lemma \ref{lem:sps_mrphsm_invert} is proved in Section \ref{sec:proof:sps_mrphsm_invert}. It relies on the observation that \eqref{eqn:def:disinteg:3} yields
	\begin{equation} \label{eqn:jsp3z}
		\graph^T(h^{-1}|\mu_X) \circ \mu_Y = \graph^T(h^{-1}|\mu_X) \circ h \circ \mu_X = \graph(h) \circ \mu_X.
	\end{equation}

	\begin{Erg}{}{} 
		In the context of $\StochCat$, Lemma \ref{lem:sps_mrphsm_invert} essentially states that a measure preserving change of variables can be inverted, after which it becomes a Markov kernel. The commutation in \eqref{eqn:sps_mrphsm_invert} essentially means that each fibre $h^{-1}(y)$ of the semi-conjugacy is invariant under the flow $\Phi^t$. Moreover, the conditional measures along these fibres are preserved under the flow, namely :
		\[ (\Phi_X^t)_* \paran{ \mu_X \mid f^{-1}(y) } = \mu_X \mid f^{-1}( \Phi_Y^t y ) , \quad \forall t\in \Time, \, \mu_Y-a.e. \, y\in T. \]
	\end{Erg}
	
	We shall find one more sufficient condition that ensures the existence of non-trivial joinings, which comes in the form of a simple patterns / diagram of $\sps$-objects.
	
	\paragraph{Joinings from v-diagrams} Consider the $\sps$ diagram :
	\begin{equation} \label{eqn:vee_diag}
		\begin{tikzcd} [column sep = small]
			\paran{ \Omega_1, \mu_1, \Phi_1 } \arrow[dr, "h_1"'] && \paran{ \Omega_2, \mu_2, \Phi_2 } \arrow[dl, "h_2"] \\
			& \paran{ \Omega_3, \mu_3, \Phi_3 }
		\end{tikzcd}
	\end{equation}
	This will be called a v-diagram. The systems $\paran{ \Omega_1, \mu_1, \Phi_1 }$ and $\paran{ \Omega_2, \mu_2, \Phi_2 }$ will be called the arms of the diagram, and the system $\paran{ \Omega_3, \mu_3, \Phi_3 }$ the apex of the diagram. We get a special state of $\Omega_3$ shown below as the composition of the two dotted, green arrows.
	\begin{equation} \label{eqn:def:disinteg_join}
		\begin{tikzcd}
			\Omega_1 && \Omega_3 \arrow[d, dotted, Shobuj] \arrow[rr, "(h_2^{-1}|\mu_2)"] \arrow[ll, "(h_1^{-1}|\mu_1)"'] && \Omega_2 \\
			&& \Omega_1 \otimes \Omega_2 \arrow[urr, "\pi_2"', bend right=10] \arrow[ull, "\pi_1", bend left=10] \\
			&& && 1 \arrow[uull, pos=0.2, "\mu_3"', dotted, Shobuj] \arrow[dashed, Shobuj]{ull}{\join(\mu_1, \mu_2 \mid h_1, h_2)}
		\end{tikzcd}
	\end{equation}
	The lower dotted, green arrow is the copied product :
	\[ \copyF \paran{ (h_1^{-1}|\mu_1) , (h_2^{-1}|\mu_2) } = \SqBrack{ (h_1^{-1}|\mu_1) \otimes (h_2^{-1}|\mu_2) } \circ \copyF_{\Omega_3} : \Omega_3 \to \Omega_1 \otimes \Omega_2\]
	We call this state 
	\[ \join(\mu_1, \mu_2 \mid h_1, h_2) := \copyF \paran{ (h_1^{-1}|\mu_1) , (h_2^{-1}|\mu_2) } \circ \mu_3 : 1 \to  \Omega_1 \otimes \Omega_2\]
	a join of $\mu_1, \mu_2$ disintegrated over the common factor $\mu_3$.
	
	\begin{lemma}[Disintegrated joins] \label{lem:disinteg_join}
		The disintegrated join constructed in \eqref{eqn:def:disinteg_join} is invariant under $\iota \Phi_1 \otimes \iota \Phi_2$ and thus leads to a join of $\paran{ \Omega_1, \mu_1, \Phi_1 }$ and $\paran{ \Omega_2, \mu_2, \Phi_2 }$.
	\end{lemma}
	
	Lemma \ref{lem:disinteg_join} is proved in Section \ref{sec:proof:disinteg_join}. It follows from Lemma \ref{lem:sps_mrphsm_invert}. We shall denote this join as $\join(\mu_1, \mu_2 \mid h_1, h_2)$.  Lemma \ref{lem:sps_mrphsm_invert}  will be used later in Section \ref{sec:join_compose} to prove a compositionality of joinings. A more immediate consequence is
	
	\begin{lemma} \label{lem:dij0s}
		\cite[Thm 3.5]{Rue2006join} If two $\sps$-objects $\mathbf{X}$ and $\mathbf{Y}$ have a non-trivial common factor, then they have a non-trivial join too.
	\end{lemma}
	
	Lemma \ref{lem:dij0s} is proved in Section \ref{sec:proof:dij0s}.
	Rudolph \cite{Rudolph1979join} showed that the converse to this conclusion is not true.
	
	\begin{definition} [Factor join]
		In the particular instance of \eqref{eqn:def:disinteg_join} in which $\paran{ \Omega_1, \mu_1, \Phi_1 } = \paran{ \Omega_2, \mu_2, \Phi_2 }$ and $h_1=h_2$, the join is called a factor join. Thus given any $\sps$ morphism $\paran{ \Omega_1, \mu_1, \Phi_1 } \xrightarrow{h} \paran{ \Omega_3, \mu_3, \Phi_3 }$ we denote
		\[\join( \mu_1 \mid h ) := \join(\mu_1, \mu_1 \mid h, h) . \]
	\end{definition}
	
	\paragraph{Projection of a joining}  
	
	\begin{lemma} \label{lem:vee_from_join}
		Given the v-diagram \eqref{eqn:vee_diag} and a join $\lambda$ of the two systems
		\[\begin{tikzcd} [column sep = small, scale cd = 0.9]
			& \paran{ \Omega_1 \otimes \Omega_2, \lambda, \Phi_1 \otimes \Phi_2 }\arrow[dr, "\pi_2"]  \arrow[dl, "\pi_1"'] \\
			\paran{ \Omega_1, \mu_1, \Phi_1 } && \paran{ \Omega_2, \mu_2, \Phi_2 } 
		\end{tikzcd} , \]
		the projection 
		\[ (h_1 \otimes h_2) \circ \lambda , \]
		of the join $\lambda$, leads to a self-joining of the third system $\paran{\Omega_3, \mu_3, \Phi_3}$.
	\end{lemma}
	
	Lemma \ref{lem:vee_from_join} is proved in Section \ref{sec:proof:vee_from_join}.  Lemma \ref{lem:vee_from_join} shows that given a v-diagram, a join of the arms lead to a self-joining of the apex system. This prompts the following sub-class of joins.
	
	\begin{definition} [Joining over a common factor] \label{def:join_over_common}
		Given the v-diagram \eqref{eqn:vee_diag}, the collection
		\[ J(\mu_1, \mu_2 \mid h_1, h_2 \mid \mu_3) := \braces{\begin{array}{c}
				\mbox{join } \lambda \mbox{ of } \mu_1, \mu_2 \mbox{ s.t. } \\
				(h_1 \otimes h_2) \circ \lambda \mbox{ is a } \\
				\mbox{the self-join } \\ \graph(\mu_3)
		\end{array}} \]
		is called the set of joinings of $\paran{ \Omega_1, \mu_1, \Phi_1 }$ and $\paran{ \Omega_2, \mu_2, \Phi_2 }$ over a common factor $\mu_3$.
	\end{definition}
	
	\begin{lemma} \label{lem:a5ukdf}
		If the common factor $\paran{ \Omega_3, \mu_3, \Phi_3 }$ is the terminal object, then the set $J(\mu_1, \mu_2 \mid h_1, h_2 \mid \mu_3) $ only contains the product join.
	\end{lemma}
	
	Lemma \ref{lem:a5ukdf} is proved in Section \ref{sec:a5ukdf}.

	\[\begin{tikzcd} [column sep = small]
		& \paran{ \Omega_3, \mu_3, \Phi_3 } \arrow[dr, "g_2"] \arrow[dl, "g_1"'] \\
		\paran{ \Omega_1, \mu_1, \Phi_1 } && \paran{ \Omega_2, \mu_2, \Phi_2 }
	\end{tikzcd}
	\begin{tikzcd} {} \arrow[rr, mapsto, "\Forget_2"] && {} \end{tikzcd} 
	\begin{tikzcd} [column sep = small]
		& \paran{ \Omega_3, \Phi_3 } \arrow[dr, "\Forget_2(g_2)"] \arrow[dl, "\Forget_2(g_1)"'] \\
		\paran{ \Omega_1, \Phi_1 } && \paran{ \Omega_2, \Phi_2 }
	\end{tikzcd}\]
	\[\begin{tikzcd} [column sep = small]
		& \paran{ \Omega_3, \Phi_3 } \arrow[dr, "\Forget_2(g_2)"] \arrow[dl, "\Forget_2(g_1)"'] \\
		\paran{ \Omega_1, \Phi_1 } && \paran{ \Omega_2, \Phi_2 }
	\end{tikzcd}
	\begin{tikzcd} {} \arrow[rr, mapsto, "\mbox{pushout}"] && {} \end{tikzcd} 
	\begin{tikzcd} [column sep = small]
		& \paran{ \Omega_3, \Phi_3 } \arrow[dr, "\Forget_2(g_2)"] \arrow[dl, "\Forget_2(g_1)"'] \\
		\paran{ \Omega_1, \Phi_1 } \arrow[dr, "f_1"'] && \paran{ \Omega_2, \Phi_2 } \arrow[dl, "f_2"] \\
		& \paran{ \Omega_4, \Phi_4 } 
	\end{tikzcd}\]

	\begin{theorem} \label{thm:join_project}
		Take two $\sps$-objects $\paran{ \Omega_1, \mu_1, \Phi_1 }$ and $\paran{ \Omega_2, \mu_2, \Phi_2 }$, and a join $\paran{ \Omega_3, \mu_3, \Phi_3 }$. Then 
		\begin{enumerate} [(i)]
			\item The pushout of the join $\paran{ \Omega_3, \mu_3, \Phi_3 }$ is a common factor $\paran{ \Omega_4, \mu_4, \Phi_4 }$ of the former two objects.
			\item $\paran{ \Omega_3, \mu_3, \Phi_3 }$ coincides with the join $\join(\mu_1, \mu_2 \mid f_1, f_2 \mid \mu_4)$ of $\paran{ \Omega_1, \mu_1, \Phi_1 }$ and $\paran{ \Omega_2, \mu_2, \Phi_2 }$ over $\mu_4$.
			\item If $\paran{ \Omega_3, \mu_3, \Phi_3 }$ is the product join iff $\paran{ \Omega_4, \mu_4, \Phi_4 }$ is the trivial terminal dynamics.
		\end{enumerate}
	\end{theorem} 
	\section{Sequence spaces} \label{sec:sequence}
	
	\paragraph{Copied products} One of the utilities of the $\copyF$ feature of Markov categories is that any finite collection of morphisms with a common domain may be bunched together to produce a joint-morphism :
	\[\left\{ \begin{tikzcd} X \arrow[d, "f_n"'] \\ Y_n \end{tikzcd} \:\, n=1,\ldots N\right\}
	\imply 
	\begin{tikzcd}
		X \arrow{d}[swap]{ \copyF^{(N)} } \arrow[drrr, dashed, Shobuj, "\otimes_{n=1}^{N, \copyF} f_n"] \\
		\otimes_{n=1}^{N} X \arrow[rrr, "\otimes_{n=1}^{N} f_n"'] &&& \otimes_{n=1}^{N} Y_n
	\end{tikzcd}\]
	Note that the morphism $\otimes_{n=1}^{N} f_n$ is a parallel grouping of the morphisms, and has domain $\otimes_{n=1}^{N} X$ which is the $N$-fold tensor product of $X$. On the other hand the morphism $\otimes_{n=1}^{N, \copyF} f_n$ has domain $X$. It represents a unified morphism in which each of the individual $f_n$-morphisms remain embedded. Recall the notion of projections from \eqref{eqn:def:projection}. It follows from the Markov identities \eqref{eqn:def:MarkovCat:1} that coordinate wise projection of a copied product recovers the individual components 
	\begin{equation} \label{eqn:d4d5z}
		\left\{ \begin{tikzcd} X \arrow[d, "f_n"'] \\ Y_n \end{tikzcd} \:\, n=1,\ldots N\right\}
		\imply 
		\begin{tikzcd}
			X \arrow{d}[swap]{ \copyF^{(N)} } \arrow[drrr, dashed, Shobuj, "\otimes_{n=1}^{N, \copyF} f_n"] \arrow[rrr, "f_i"] &&& Y_i \\
			\otimes_{n=1}^{N} X \arrow[rrr, "\otimes_{n=1}^{N} f_n"'] &&& \otimes_{n=1}^{N} Y_n \arrow[u, "\pi_i"']
		\end{tikzcd} , \; \forall i\in 1,\ldots,N.
	\end{equation}

	\begin{Erg}{}{} 
		In any Cartesian category such as $\SetCat$, $\Topo$ and $\MeasCat$, the copied product is simply the categorical product of the morphisms. 
	\end{Erg}
	
	\begin{Erg}{}{} 
		The copied product $\otimes_{n=1}^{N, \copyF} f_n$ in $\StochCat$ represent the joint distribution on the domain $\otimes_{n=1}^{N} Y_n$ induced by the Markov kernels $f_n$. More importantly, the kernels $f_n$ are present independently within the morphism $\otimes_{n=1}^{N} f_n$ while they are jointly distributed within $\otimes_{n=1}^{N, \copyF} f_n$.
	\end{Erg}
	
	\paragraph{Infinite powers} Fix a set $S$ for the rest of this chapter. Then one has a poset created from the collection of subsets of $S$ :
	\[ 2^{S} := \text{ poset of subsets of } S .\]

	\begin{definition} [Co-filtration] \label{def:cofltr}
		An $S$-indexed co-filtration is an member of the following functor category :
		\[ \CoFilter(\calC; S) := \Functor{ \paran{2^S}^{op} }{ \calC }. \]
	\end{definition}
	
	\begin{definition} [Power induced co-filtration] \label{def:power_cofilter}
		For any object $\Omega$ of $\calC$, the  $S$-indexed co-filtration created by powers of $\Omega$ is the functor 
		\[ \Pow_{\Omega, S} \in \CoFilter(\calC; S), \, 
		\begin{tikzcd} A \arrow[d, "\subset"'] \\ A' \end{tikzcd} \,\mapsto\, 
		\begin{tikzcd} \otimes_{|A'|} \Omega \arrow[d, "\pi_{A'\to A}"'] \\ \otimes_{|A|} \Omega \end{tikzcd}\]
	\end{definition}
	
	\begin{definition} [Co-filtered limit]
		An object $X$ of $\calY$ is said to have a cofiltered limit of order $S$ for some indexing set $S$, if $S$-indexed co-filtration $\Pow_{X, S}$ created by powers of $\Omega$ has a limit. This limit is denoted as $\otimes^{|S|} X$.
	\end{definition}
	
	\begin{definition} [Shift spaces] \label{def:shift}
		Suppose an object $X$ of $\calY$ has a cofiltered limit indexed by $\Time$. Then there is a natural dynamics with domain $\otimes^{|\Time|} X$ given by 
		\begin{equation} \label{eqn:def:shift}
			\begin{tikzcd}
				\otimes^{|J+t|} X & & & & \otimes^{|J|} X \arrow[llll, "+t"'] \\
				& \otimes^{|\Time|} X \arrow[ld, "\pi_{J'+t}"] \arrow[lu, "\pi_{J+t}"'] & & \otimes^{|\Time|} X \arrow[ll, "\shift^t"] \arrow[rd, "\pi_{J'}"'] \arrow[ru, "\pi_J"]  & \\
				\otimes^{|J'+t|} X \arrow[uu, "\pi_{J'+t \to J+t}"] & & & & \otimes^{|J'|} X \arrow[uu, "\pi_{J'\to J}"'] \arrow[llll, "+t"]
			\end{tikzcd} ,
			\quad \forall J\subseteq J' \subset_{fin} \Hom(\Time).
		\end{equation}
		This dynamics is called a \emph{shift} or a \emph{shift-flow}.
	\end{definition} 
	
	\begin{lemma} [Shift-space dynamics] \label{lem:shift}
		Suppose an object $X$ of $\calY$ has a cofiltered limit indexed by $\Time$. Then the shift space from Definition \ref{def:shift} is a dynamical system in $\calX$. Take any morphism $\mu : 1\to X$. Then 
		\begin{enumerate} [(i)]
			\item there is a unique morphism $\otimes^{|\Time|} \mu : 1 \to \otimes^{|\Time|} X$ such that
			\[\begin{tikzcd}
				& & & \otimes^{|J|} X \\
				1 \arrow[rrru, "\otimes^{|J|} \mu", bend left] \arrow[rrrd, "\otimes^{|J'|} \mu"', bend right] \arrow[rr, "\otimes^{|\Time|} \mu", dotted] & & \otimes^{|\Time|} X \arrow[rd, "\pi_{J'}"'] \arrow[ru, "\pi_J"] & \\
				& & & \otimes^{|J'|} X \arrow[uu, "\pi_{J'\to J}"']
			\end{tikzcd},
			\quad \forall J\subseteq J' \subset_{fin} \Hom(\Time).\]
			\item This morphism $\otimes^{|\Time|} \mu$ is invariant under the shift.
		\end{enumerate}
	\end{lemma}
	
	Lemma \ref{lem:shift} is proved in Section \ref{sec:proof:shift}.
	
	\begin{definition} [Infinite copied products] \label{def:inf_copy}
		Suppose an object $X$ of $\calY$ has a cofiltered limit indexed by a set $S$. Suppose there is a $\calY$-morphism $f:Y \to X$. Then there is a unique morphism
		\[ \copyF^{|S|}(f) : Y \to \otimes^{|S|} X , \]
		such that the following commutations holds 
		\[\begin{tikzcd}
			& & & \otimes^{|J|} X \\
			Y \arrow[rrru, "\copyF^{|J|}(f)", bend left] \arrow[rrrd, "\copyF^{|J'|}(f)"', bend right] \arrow[rr, "\copyF^{|S|}(f)", dotted] & & \otimes^{|\Time|} X \arrow[rd, "\pi_{J'}"'] \arrow[ru, "\pi_J"] & \\
			& & & \otimes^{|J'|} X \arrow[uu, "\pi_{J'\to J}"']
		\end{tikzcd},
		\quad \forall J\subseteq J' \subset_{fin} \Hom(\Time).\]
	\end{definition}
	
	\begin{definition} [Infinite copies of states]
		Suppose that an object $X$ of the Markov category $\calY$ has tensor products of order $|S|$ for some indexing set $S$. Recall the conditional morphism $\lambda_{Y|X}$ from \eqref{eqn:dij03}.
		\begin{equation} \label{eqn:def:mu_inf}
			\begin{tikzcd}
				1 \arrow[d, "\mu_Y"'] \arrow[rr, dashed, Shobuj, "\mu_{Y\to X}^{|S|}"] && \otimes^{|S|} X \\
				Y \arrow{urr}[swap]{ \copyF^{|S|}(\lambda_{Y|X}) }
			\end{tikzcd} , \quad 
			\begin{tikzcd}
				1 \arrow[d, "\mu_Y"'] \arrow[rr, dashed, Shobuj, "\tilde{\mu}_{Y\to X}^{|S|}"] && \paran{\otimes^{|S|} X} \otimes Y \\
				Y \arrow[rr, "\copyF_Y"'] && Y \otimes Y \arrow{u}[swap]{\copyF^{|S|}(\lambda_{Y|X}) \otimes Y }
			\end{tikzcd}
		\end{equation}
	\end{definition}

	In conclusion, starting with the join / span on the left, we end with the construction on the right :
	\[ \begin{tikzcd} (X, \mu_X) \\ ( X \otimes Y, \lambda ) \arrow[d, "\pi_Y"'] \arrow[u, "\pi_X"] \\ (Y, \mu_Y) \end{tikzcd}
	\imply
	\begin{tikzcd} [column sep = large]
		& & 1 \arrow[d, "\mu_Y"] & & \\
		\otimes^{|S|} X & & Y \arrow[ll, pos=0.7, "\lambda_{Y\to X}^{|S|}"'] \arrow[lld, pos=0.7, "\tilde{\lambda}_{Y\to X}^{|S|}"'] \arrow[d, "="] \arrow[rrd, pos=0.7, "\lambda_{Y|X}"] \arrow[rr, pos=0.7, "\lambda_{Y\to X}"] & & X \\
		\left( \otimes^{|S|} X \right) \otimes Y \arrow[u, Holud, "\pi_1"] \arrow[rr, Holud, "\pi_2"'] & & Y & & X\otimes Y \arrow[ll, "\pi_Y"] \arrow[u, "\pi_X"']
	\end{tikzcd}\]

	\begin{definition} [Permutations]
		Suppose that an object $X$ of the Markov category $\calY$ has co-filtered limits of order $|S|$ for some indexing set $S$. The any bijection $\sigma : S \xrightarrow{\cong} S$ of $S$ leads to an induced automorphism, which we also denote as $\sigma$. Such as automorphism is called a permutation in the coordinates of $\otimes^{|S|} X$.
	\end{definition}
	
	\begin{lemma} [Permutation invariance] \label{lem:f4jd93}
		\cite[Lem 9]{LemanczykEtAl1993gaussian} Consider the pointed object $\paran{ \tilde{\mu}_{Y\to X}^{|S|}, \paran{\otimes^{|S|} X} \otimes Y }$ constructed as in \eqref{eqn:def:mu_inf}. If there is a morphism $g : \paran{\otimes^{|S|} X} \otimes Y \to \calR$ which remains invariant under all finite permutations in the coordinates of $\otimes^{|S|} X$, then $g$ factors through the deletion morphism $\pi_Y : \otimes^{|S|} X \otimes Y \to Y$.
	\end{lemma}
	
	Lemma \ref{lem:f4jd93} thus says that the following group action
	\[ \begin{array}{cc}
		\mbox{finite}  \\
		\mbox{permutations}
	\end{array} 
	\begin{tikzcd}
		{} \arrow[rr, mapsto] && {}
	\end{tikzcd}
	\Endo \paran{ X^S \times Y }\]
	has as colimit the projection morphism :
	\[ \proj_Y : X^S \times Y \to Y. \]

	\begin{lemma} \label{lem:Thm4:b}
		\cite[Thm 4]{LemanczykEtAl1993gaussian} Consider the pointed object $\paran{ \tilde{\mu}_{Y\to X}^{|S|}, \paran{\otimes^{|S|} X} \otimes Y }$ constructed as in \eqref{eqn:def:mu_inf}. Then the composite morphism
		\[\begin{tikzcd}
			X^\infty \arrow[rrr, "\left( \lambda^{\infty} | \pi_{X_{\infty}}^{-1} \right)"] &&& X^\infty \otimes Y \arrow[r, "\pi_Y"] & Y
		\end{tikzcd}\]
		is invariant under finite permutations in the coordinates of $X^\infty$.
	\end{lemma}

	\begin{theorem} [Push-pull theorem] \label{thm:push_pull}
		Consider a joining in the slice category $\Comma{1}{\iota}$, as shown on the left below : 
		\begin{equation} \label{eqn:thm:push_pull}
			\begin{tikzpicture} \node[draw, inner sep=5pt, draw=ChhaiD, line width=2pt] (box){
					\begin{tikzcd} [scale cd = 0.7, column sep = small]
						& \paran{\Omega_X \otimes \Omega_Y, \lambda} \arrow[dr, "\pi_Y"] \arrow[dl, "\pi_X"'] \\
						\paran{\Omega_X, \mu_X} && \paran{\Omega_Y, \mu_Y}
					\end{tikzcd}
				}; \end{tikzpicture}
		\end{equation}
		Moreover, the object $\paran{ \tilde{\mu}_{Y\to X}^{|S|}, \paran{\otimes^{|S|} X} \otimes Y }$ is the unique cone above the v-diagram created by the right half of this picture whose second component is $\paran{\otimes^{|S|} X} \otimes Y$.
	\end{theorem}
	
	Using the conditional expector \eqref{eqn:lem:subsig:1} we can construct the following states on $\Omega_X\otimes \Omega_Y$ and $\otimes^{\Time} \Omega_X \otimes \Omega_Y$, shown below as green dashed arrows :
	\[\begin{tikzcd}
		\Omega_Y \arrow[d, "\copyF_{\Omega_Y}"'] &&& 1 \arrow[lll, "\mu_Y"'] \arrow[Shobuj, dashed, d, "\lambda"] \\
		\Omega_Y\otimes \Omega_Y \arrow{rrr}{ \CE{\Omega_X}{\Omega_Y}{\lambda}\otimes \Omega_Y } &&& \Omega_X\otimes \Omega_Y
	\end{tikzcd} ,
	\quad
	\begin{tikzcd}
		\Omega_Y \arrow[d, "\copyF^{(|\Time|+1)}_{\Omega_Y}"'] &&&&& 1 \arrow[lllll, "\mu_Y"'] \arrow[Shobuj, dashed, d, "\lambda_{\Time}"] \\
		\Omega_Y\otimes \paran{ \otimes^{\Time} \Omega_Y } \arrow{rrrrr}{ \copyF^{(\Time} \SqBrack{ \CE{\Omega_X}{\Omega_Y}{\lambda} } \otimes \Omega_Y } &&&&& \Omega_X\otimes \paran{ \otimes^{\Time} \Omega_X }
	\end{tikzcd}\]
	The diagram in \eqref{eqn:thm:push_pull} can be redrawn with these states as follows :
	\[\begin{tikzpicture} \node[draw, inner sep=5pt, draw=ChhaiD, line width=2pt] (box){
			\begin{tikzcd}
				&  &  & \otimes^{\Time} \Omega_X \otimes \Omega_Y \arrow[rdd, "\pi_Y", bend left] \arrow[llddd, pos=0.7, bend right=5, "\pi_{(\otimes^{\Time} \Omega_X)}"'] &  \\
				1 \arrow[ddr, Itranga, "\mu_{X,\Time}"'] \arrow[rrr, Itranga, "\lambda"] \arrow[rrru, Itranga, "\lambda_{\Time}"] &  &  & \Omega_X \otimes \Omega_Y \arrow[rd, "\pi_Y"] \arrow[ld, "\pi_X"'] \arrow[u, "\iota_t \otimes \Omega_Y"'] &  \\
				&  & \Omega_X \arrow[rr, dotted, Akashi, "\CE{Y}{X}{\lambda}"] &  & \Omega_Y \arrow[dr, Holud, "\psi"] \\
				& \otimes^{\Time} X \arrow[ur, dotted, Akashi, "\pi_t"'] \arrow[rrru, dotted, bend right=5, Akashi, "\CE{Y}{\otimes^{\Time} X}{\lambda_{\Time}}"'] &  &  & & \Holud{\calZ}
			\end{tikzcd}
		}; \end{tikzpicture}\]
	The blue dotted arrows in these diagrams are again conditional expectors. They do not create a strict commutation but only almost-sure commutation. 
	Then $G$ is invariant under all permutations of coordinates in $X$ and hence by Lemma \ref{lem:f4jd93}, it factors through $\pi_2 : \paran{\otimes^{|S|} X} \otimes Y \to Y$. 
	
	\begin{lemma} [Dynamics extend to colimits of domains] \label{lem:dyn_colim}
		Suppose $\psi : J\to \calC$ is a diagram in a category $\calC$ with a colimit $X$. Suppose that there is a functor $J$-diagram in the category of dynamical systems in $\calC$ whose domains realize the diagram $\psi$. In other words the following commutation diagram exists :
		\[\begin{tikzcd} [column sep = large]
			J \arrow[r, "F"] \arrow[dr, dashed, "\psi"'] & \Functor{\Time}{\calC} \arrow[d, "\dom"] \\
			& \calC
		\end{tikzcd}\]
		Then there is a unique dynamical system $\bar{\Phi}$ on $X$ such that the colimiting cone $\eta : \psi \Rightarrow X$ serves as a co-cone for $F$.
	\end{lemma}
	
	Lemma \ref{lem:dyn_colim} is proved in Section \ref{sec:proof:dyn_colim}.
	
	\begin{corollary} [Relative independence theorem] \label{cor:rel_ind}
		\cite[Thm 3.6]{Rue2006join} \cite{GlasnerWeiss2003past} \cite[Thm 4]{LemanczykEtAl1993gaussian} If two $\sps$-objects $\mathbf{X}$ and $\mathbf{Y}$ have a non-trivial join, then $\mathbf{Y}$ has a common factor with some countable self-join of $\mathbf{X}$.
	\end{corollary}
	
	
	\begin{corollary} [Product of disjoint systems is a categorical product] \label{lem:disjoint_lim}
		Two $\sps$ objects $(\Omega_1, \mu_1, \Phi_1)$ and $(\Omega_2, \mu_2, \Phi_2)$ are disjoint iff they have a categorical product, and this categorical product is their trivial join.    
	\end{corollary}
	
	\begin{definition} [Arithmetic completeness]
		A subset of $\sps$-objects is said to be \textit{arithmetically complete} / has \textit{arithmetic completeness} if it is closed under taking factors and countable joins.
	\end{definition}
	
	\begin{corollary} \label{cor:arith_cmplt}
		Suppose that $\calS$ is an arithmetically complete subset of $\sps$. Then the two sets
		\begin{equation}
			\begin{split}
				\calS^\bot &:= \SetDef{\mathbf{Y} \in\sps}{ \mathbf{Y} \bot \mathbf{X} , \forall \mathbf{X}\in\calS} \\
				\calS^{\text{join}, c} &:= \SetDef{\mathbf{Y} \in\sps}{ \mathbf{Y} \mbox{ has no non-trivial morphism to } \mathbf{X} , \, \forall \mathbf{X}\in\calS}
			\end{split}
		\end{equation}
		are the same. In other words, given an arithmetically complete collection $\calS$ of $\sps$-objects, the collection of $\calS$-disjoint objects coincide with the collection of objects which are $\calS$-factor free.
	\end{corollary}
	
	Corollary \ref{cor:arith_cmplt} is proved in Section \ref{sec:proof:arith_cmplt}. The proof is a simple application of Theorem \ref{thm:push_pull} and Lemma \ref{lem:dij0s}. Corollary \ref{cor:arith_cmplt} has immense applications in Ergodic theory, where various arithmetically complete collections have been identified, each signifying different dynamic properties.
	
	\begin{Erg}{}{} 
		A measure-preserving dynamical system $(X, \mu, \Phi)$ is called a \textit{zero-entropy system} if its metric entropy / Kolmogorov-Sinai metric entropy $h_\mu(\Phi)$ is zero. Metric entropy quantifies the average rate at which a system generates complexity by the separation of trajectories over time. Therefore, a system with zero entropy lacks the exponential divergence and inherent unpredictability, often interpreted as the broad concept of "chaos". This makes zero-entropy systems highly predictable; their future states can be forecasted with sub-exponential error growth given a finite precision of the initial state. Classic examples include identity maps, periodic transformations, and irrational rotations on a circle. While they do not exhibit strong mixing properties, the study of zero-entropy systems is remarkably rich, relying heavily on spectral theory and algebraic invariants to classify structured, deterministic behaviors \cite[e.g.]{dooley2005approx}. They also form a family of dynamical systems which are  arithmetically complete \cite[e.g.]{glasner1995quasi, blanchard1993zero}. 
	\end{Erg}
	
	\begin{Erg}{}{} 
		A measure-preserving dynamical system is said to have a discrete spectrum (or pure point spectrum) if the Hilbert space of square-integrable functions over its state space is spanned entirely by the eigenfunctions of its associated Koopman operator. Such systems are fundamentally characterized by their regular, highly predictable, and rotation-like dynamics, lacking the strong mixing properties found in chaotic systems. The landmark Halmos-von Neumann representation theorem established that ergodic systems with a discrete spectrum are isomorphic to group translations on compact abelian groups, completely classifying them up to measure-theoretic isomorphism. Structurally, it has been shown that ergodic transformations with a discrete spectrum belong to the class of stacking transformations, meaning they can be explicitly constructed via cutting and stacking methods using single stacks \cite{DelJunco1976}. Modern investigations have expanded these classifications into topological dynamics, demonstrating that factors of totally ergodic systems with quasi-discrete spectra inherently possess zero entropy \cite{HaaseMoriakov2018}. Discrete-spectrum systems also form a family of dynamical systems which are  arithmetically complete \cite[e.g.]{Furstenberg1967}
	\end{Erg}
	
	\begin{Erg}{}{} 
		A measure-preserving dynamical system $(X, \mathcal{B}, \mu, \Phi)$ is \textit{weakly mixing} if it exhibits mixing in the sense of Cesàro averages, meaning that for any measurable sets $A, B \in \mathcal{B}$, $\lim_{n \to \infty} \frac{1}{n} \sum_{k=0}^{n-1} \left\vert{} \mu(\Phi^{-k}(A) \cap B) - \mu(A)\mu(B) \right\vert{} = 0$. From a spectral perspective, $\Phi$ is weakly mixing if and only if its associated Koopman operator on $L^2(X, \mu)$ has a purely continuous spectrum on the orthogonal complement of the constant functions, meaning it possesses no non-trivial eigenfunctions. 
	\end{Erg}
	
	\begin{Erg}{}{} 
		In the context of ergodic theory and topological dynamics, an isometric system (or more broadly, an isometric extension) refers to a measure-preserving dynamical system where the transformation acts as an isometry along the fibers of the extension with respect to some compatible metric \cite{glasner1978exist}. Such systems are intrinsically rigid, exhibiting regular, predictable behavior and zero Kolmogorov-Sinai entropy along the fibers. These structured extensions are fundamental building blocks of dynamical systems; for instance, the rigorous classification of isometric extensions of multidimensional Bernoulli shifts helps mathematically isolate completely deterministic factors from purely mixing ones (Kammeyer, 1992). In the measure-theoretic framework, isometric extensions are closely associated with systems possessing relative discrete spectrums, which stand in sharp contrast to purely chaotic, weakly mixing extensions, though they can still support unique equilibrium measures when extending strongly chaotic Anosov systems (Spatzier and Visscher, 2016). Furthermore, the structural rigidity of isometric extensions is frequently analyzed using joining theory, where properties like quasi-simplicity dictate that any ergodic non-product self-joining of the system must be strictly isometric over its marginals (Danilenko, 2007). Such joining-based classifications also play a critical role in distinguishing positive-entropy, "dominant" systems from those exhibiting distal or isometric constraints (Austin et al., 2022).
	\end{Erg}
	
	\begin{Erg}{}{} 
		A Kolmogorov automorphism, or K-system, is a measure-preserving dynamical system characterized by having completely positive entropy \cite{CallPark2021k}. This means that every non-trivial factor of the system has strictly positive Kolmogorov-Sinai entropy, lacking any deterministic or zero-entropy components. Formally, a system is a K-automorphism if it possesses a sub-$\sigma$-algebra—often called a K-algebra—that grows to generate the entire $\sigma$-algebra under the forward transformation and shrinks to the trivial $\sigma$-algebra under the inverse, meaning its associated Pinsker algebra is trivial. Because of this structure, K-systems sit relatively high in the ergodic hierarchy, being strictly stronger than strongly mixing systems and weaker than Bernoulli shifts. 
		The K-property is a hallmark of robust chaos and has been verified for certain types of dynamics \cite[e.g.]{ChernovHaskell1996}. The structural rigidity of these chaotic systems is frequently analyzed using joining theory; for instance, joinings can be used to prove that systems with completely positive entropy are fundamentally disjoint from all zero-entropy systems \cite{RyzhikovThouvenot2024quasi}. Furthermore, relatively independent joinings serve as essential tools to dissect characteristic factors and limit behaviors in extensions of systems possessing varying degrees of mixing and entropy \cite{ackelsberg2025khintchine}.
		\cite{Thouvenot1995join}
	\end{Erg}
	
	We next take a deeper dive into the property of ergodicity.
	
	\section{Ergodicity} \label{sec:erg}
	
	\begin{definition} [Ergodicity]
		\cite[Def 3.8]{MossPerrone2022ergdc} An invariant state $p:1\to \iota \Omega$ of a dynamical system $\Phi:\Time \to \calX$ is said to be ergodic if for every invariant deterministic observable $c:\iota \Omega \to \calR$, the composite $c\circ p : 1\to \calR$ is deterministic.
	\end{definition}
	
	Note that ergodicity is not a property of the dynamics $\Phi$ itself but of an invariant state $p$ of the dynamics. Broadly, it means that any invariant observable $c$ of the dynamical system must be constant. The notion of a "constant" itself is hard to capture in any other manner in a Markov category.
	
	\begin{definition} [Identity system]
		An identity system in category $\calC$ is a trivial functor $F : \star \to \calC$ from the trivial 1-object category $\star$. Such systems correspond bijectively to the objects of $\calC$. If $A$ denotes the one and only object in the image of $F$, then such a system may be interpreted as the dynamics $\Id_A : A \to A$.
	\end{definition}
	
	Note that any semigroup such as $\calT$ has a unique homomorphism into $\star$. This can be used to pullback any identity system as follows :
	\[\begin{tikzcd}
		\calC \arrow[r, "\cong"] & \Functor{\star}{\calC} \arrow[rr, "!_{\Time\to\star} \circ", hook] && \Functor{\Time}{\calC}
	\end{tikzcd} .\]
	Thus any identity system may be artificially interpreted as a system with time structure $\Time$. Now note that if the category $\calC$ has a terminal element $1$, then any dynamical system bears a homomorphism with an identity system. Let us take a general dynamics $\Phi:\Time\to\calC$ with domain $\Omega$, and aa stationary system $\Id : \Time\to \calC$ with domain $A$. Fix a state /element $a:1\to A$. Then we have commutations
	\[\begin{tikzcd} [column sep = large]
		A \arrow[r, "!"] \arrow[d, "\Phi^t"'] & 1 \arrow[d, "\Id"] \arrow[r, "a"] & A \arrow[d, "\Id"] \\
		A \arrow[r, "!"'] & 1 \arrow[r, "a"'] & A 
	\end{tikzcd} , \quad \forall t\in \Time.\]
	The second column is the identity system with domain $1$. The third column is the identity system on $A$. 
	
	\begin{definition} [Irreducible dynamics]
		A dynamical system $\Phi:\Time\to\calC$ with domain $\Omega$ is said to be irreducible if it has no non-trivial morphism to an identity system.
	\end{definition}
	
	\begin{lemma} \label{lem:d90did}
		If there is an $\sps$-morphism from an identity system to an $\sps$ object $X$, then $X$ must be an identity system itself.
	\end{lemma}
	
	\begin{definition} [Ergodic extensions]
		Let $\calC$ be any category, and consider two dynamical systems $\Phi, \Psi : \Time \to \calC$. A homomorphism of dynamical systems $\pi : \Phi \Rightarrow \Psi$ is said to be an ergodic extension if any $\Phi$-invariant observable can be factored through a $\Psi$-invariant observable. In other words, if $\Phi, \Psi : \Time \to \calC$ are two dynamical systems with domains $\Omega, \Lambda$ respectively, and there is a semi-conjugacy $\pi$ as shown below on the left :
		\begin{equation} \label{eqn:def:erg_extn}
			\begin{tikzcd} [row sep = large]
				\Omega \arrow[dd, "\pi"'] \arrow[rr, "\Phi^t"] & & \Omega \arrow[dd, "\pi"] \\
				\\
				\Lambda \arrow[rr, "\Psi^t"'] & & \Lambda
			\end{tikzcd}
			\begin{array}{c} \mbox{is an} \\ \mbox{ergodic} \\ \mbox{extension} \end{array}
			\;\Leftrightarrow\;
			\begin{tikzcd} [row sep = large]
				\Omega \arrow[dr, "\alpha"] \arrow[dd, "\pi"'] \arrow[rr, "\Phi^t"] & & \Omega \arrow[dd, "\pi"] \arrow[dl, "\alpha"'] \\
				& Y\\
				\Lambda \arrow[rr, "\Psi^t"'] & & \Lambda
			\end{tikzcd} 
			\,\Rightarrow\,
			\exists \begin{tikzcd} \Lambda \arrow[d, "\beta"] \\ Y \end{tikzcd} \,:\,
			\begin{tikzcd} [row sep = large]
				\Omega \arrow[dr, "\alpha"] \arrow[dd, "\pi"'] \arrow[rr, "\Phi^t"] & & \Omega \arrow[dd, "\pi"] \arrow[dl, "\alpha"'] \\
				& Y\\
				\Lambda \arrow[ur, "\beta"] \arrow[rr, "\Psi^t"'] & & \Lambda \arrow[ul, "\beta"']
			\end{tikzcd} 
		\end{equation}
		then any $\Phi$-invariant observable $\alpha$ as shown in the middle, must factor through a $\Psi$-invariant observable $\beta$ as shown on the right.
	\end{definition} 
	
	\begin{lemma} [Ergodicity of extensions of identity] \label{lem:0dj9xc}
		\cite[Lem 6.24]{glasner2015join} Take a morphism $\phi : \paran{\Omega, \mu, \Phi} \to (Z, \eta, \Id_Z)$ to an identity system. Then $\pi$ is ergodic iff for any morphism of the form $\pi : (X, \nu, \Id_X) \to (Z, \eta, \Id_Z)$, $\paran{\Omega, \mu, \Phi}$ and $(X, \nu, \Id_X)$ over their common factor $(Z, \eta, \Id_Z)$.
	\end{lemma}
	
	Lemma \ref{lem:0dj9xc} is proved in Section \ref{sec:proof:0dj9xc}. Lemma \ref{lem:0dj9xc} will be key in establishing equivalent characterizations of ergodicity later in Theorem \ref{thm:9dj8zl}.
	
	\begin{definition} [Birkhoff property]
		An $\sps$ object $(X, \mu, \Phi)$ has the \emph{Birkhoff property} if the following condition is satisfied: If a state $h$ with the marginal for $X$ being is invariant under $\Phi^t \times 1_Y$, i.e., if the diagrams below on the left coincide for every $t\in \Time$
		
		\begin{center}
			$\forall t\in \Time$ \,:\,
			\begin{tikzpicture}
				\node[box=1/0/1/0] (f) at (1,-1) {\Phi^t};
				\node[box=1/0/1/0] (g) at (3,-1) {1_Y};
				\node[box=2/0/1/0] (h) at (2,-3) {h};
				\wires{
					f = {south = h.north.1},
					g = {south = h.north.2},
				}{f.north, g.north}
			\end{tikzpicture}
			\mbox{coincide}
			\begin{tikzcd} {} \arrow[rr, Rightarrow] && {} \end{tikzcd}
			\begin{tikzpicture}
				\node[box=1/0/1/0] (mu) at (1,-1) {\mu};
				\node[box=1/0/1/0] (g) at (3,-1) {1_Y};
				\node[dot] (n) at (1,-2) {};
				\node[dot] (n2) at (1,-2.2) {};
				\node[box=2/0/1/0] (h) at (2,-3) {h};
				\wires{
					mu = {south = n.north},
					n2 = {south = h.north.1},
					g = {south = h.north.2},
				}{f.north, g.north}
			\end{tikzpicture}
		\end{center}
		then these diagrams coincide with the one on the right.
	\end{definition}
	
	\begin{theorem} \label{thm:fp9s}
		For a dynamical system $(X, f, \mu)$, where $f$ is deterministic, the following properties hold:
		\begin{enumerate}
			\item If there is only one joining between $(X,f,\mu)$ and $(Y, 1_Y, \nu)$ for all $\nu$, $(X,f,\mu)$ is ergodic.
			\item If $(X, f, \mu)$ satisfies Birkhoff property, there is only one joining between $(X,f,\mu)$ and $(Y, 1_Y, \nu)$ for all $\nu$.
		\end{enumerate}
	\end{theorem}
	
	Theorem \ref{thm:fp9s} is proved in Section \ref{sec:proof:fp9s}.
	
	\begin{theorem} [Equivalent characterization of ergodicity] \label{thm:9dj8zl}
		Let $\mu$ be an invariant state in $\calY$ of a dynamical system $\Phi : \Time \to \calX$ with domain $\Omega$. Then the following are equivalent
		\begin{enumerate} [(i)]
			\item $\mu$ is an ergodic state of $(\iota\Omega, \iota\Phi)$;
			\item the $\sps$-object $\paran{\Omega, \mu, \Phi}$ is irreducible;
			\item the $\sps$-object $\paran{\Omega, \mu, \Phi}$ is disjoint from every identity system;
			\item the terminal morphism $\paran{\Omega, \mu, \Phi} \xrightarrow{!} \paran{1, 1, \Id}$ is an ergodic extension.
		\end{enumerate}
	\end{theorem}
	\section{Composition of joinings and enrichment} \label{sec:join_compose}
	
	Joinings can be composed. In classical probability theory, and optimal transport, this process is formally justified by the \textit{gluing Lemma}, which relies heavily on the disintegration property. Given a setup as shown below
	\[\begin{tikzcd}
		& (X_1 \otimes X_2, \lambda_{1,2} ) \arrow[dl, "\pi^{(1,2)}_1"'] \arrow[dr, "\pi^{(1,2)}_2"] & & (X_2 \otimes X_3, \lambda_{2,3} ) \arrow[dl, "\pi^{(2,3)}_2"'] \arrow[dr, "\pi^{(2,3)}_3"] \\
		(X_1, \mu_1) & & (X_2, \mu_2) & & (X_3, \mu_3)
	\end{tikzcd}\]
	we can compose or glue these joins to get a joining of $(X_1, \mu_1) $ and $(X_3, \mu_3)$. For that, we first use disintegration of measures to obtain the morphisms  $\left( \pi^{(1,2)-1}_2 | \mu_2 \right)$ and $\left( \pi^{(2,3)-1}_2 | \mu_2 \right)$ as shown below :
	\[ \exists \begin{tikzcd} X_2 \arrow[d, "\left( \pi^{(1,2)-1}_2 | \mu_2 \right)"] \\ X_1 \otimes X_2 \end{tikzcd} 
	; \quad 
	\exists \begin{tikzcd} X_2 \arrow[d, "\left( \pi^{(2,3)-1}_2 | \mu_2 \right)"] \\ X_2 \otimes X_3 \end{tikzcd}\]
	The marginals of these disintegrations, obtained by composing with the marginal projection, may be interpreted as a categorical analog of conditional expectation :
	\[ \begin{tikzcd} 
		X_2 \arrow[d, "\left( \pi^{(1,2)-1}_2 | \mu_2 \right)"'] \arrow[dashed, Shobuj]{rr}{\mathbb{E}(X_1|X_2, \lambda_{1,2})} && X_1 \\ 
		X_1 \otimes X_2 \arrow[urr, "\pi^{(1,2)}_1"'] 
	\end{tikzcd} 
	; \quad 
	\begin{tikzcd} 
		X_2 \arrow[d, "\left( \pi^{(2,3)-1}_2 | \mu_2 \right)"'] \arrow[dashed, Shobuj]{rr}{\mathbb{E}(X_3|X_2, \lambda_{2,3})} && X_3 \\ 
		X_2 \otimes X_3 \arrow[urr, "\pi^{(2,3)}_3"']
	\end{tikzcd}\]
	Finally, we collect these conditional expectations by taking their monoidal product, to obtain a jont state on $\Omega_1 \otimes \Omega_3$ :
	\begin{equation} \label{eqn:def:join_compose}
		\begin{tikzcd}
			X_2 \arrow[rr, "\copyF_{X_2}^{(2)}"] && X_2 \otimes X_2 \arrow{d}{ \mathbb{E}(X_1|X_2, \lambda_{1,2}) \otimes \mathbb{E}(X_3|X_2, \lambda_{2,3}) } \\
			1 \arrow[u, "\mu_2"] \arrow[Shobuj, dashed, rr, "\lambda_{1,2} \circ \lambda_{2,3}"'] && X_1 \otimes X_3
		\end{tikzcd}
	\end{equation}
	This state $\lambda_{1,2} \circ \lambda_{2,3}$ can be shown to be a join of $\mu_1$ and $\mu_3$, by way of the following commutation :
	\begin{equation} \label{eqn:join_compose:1}
		\begin{tikzcd} [column sep = large]
			&& X_1 \\
			1 \arrow[Shobuj, dashed, rr, "\lambda_{1,2} \circ \lambda_{2,3}"'] \arrow[dashed, bend left=10, urr, "\mu_1"] \arrow[dashed, bend right=10, drr, "\mu_3"'] && X_1 \otimes X_3 \arrow[d, "\pi_3"] \arrow[u, "\pi_1"'] \\
			&& X_3
		\end{tikzcd}
	\end{equation}
	The identity \eqref{eqn:join_compose:1} is proved in Section \ref{sec:proof:join_compose:1}. 
	Thus we have obtained a category theoretic version of the \emph{Gluing lemma} of optimal transport theory \cite[e.g.]{Villani2009opt}, \cite[Lem 7.6]{villani2021topics}. This mechanism is central to proving the triangle inequality for Wasserstein metrics and establishing the foundational topology for spaces of measures.
	
	\begin{theorem} [Compositionality of joinings] \label{thm:joining_compose}
		The composition rule \eqref{eqn:def:join_compose} for joinings is associative.
	\end{theorem}
	
	The Disintegration Theorem only guarantees that the families of conditional measures (the Markov kernels) are unique $\mu$-almost everywhere. This means that $K_{12}(x_1, \cdot)$ could theoretically be altered on a set of $\mu_1$-measure zero without changing the static joining $\lambda_{12}$. However, this does not break the associativity of the joinings themselves. Because the final step of composition always requires integrating the kernels back against the base marginals (e.g., $\mu_1$), any alterations made on a set of measure zero are integrated out. Therefore, while the intermediate kernels are only defined up to an almost-everywhere equivalence class, the final composed joining—the actual measure on $X_1 \times X_4$—is strictly unique and strictly associative. 
	
	\paragraph{Enrichment} Perrone \cite{Perrone2024entropy} presented an interesting proposition of interpreting the category $\StochCat$ to be enriched over the category of \emph{divergence spaces}. Our observation that the set of joinings are compositional in nature allow for enrichment over more structured categories that can hold more information about the joinings. Here are some examples
	\begin{enumerate}
		\item For any two probability spaces $(X, \mu)$ and $(Y, \nu)$, the set of all joinings is a convex set: if $\lambda_1, \lambda_2 \in \join(\mu, \nu)$, then any convex combination $t\lambda_1 + (1-t)\lambda_2$ (for $t \in [0, 1]$) is also a valid joining. Let $\ConvexCat$ be the category of convex spaces whose morphisms are affine maps that preserve convex combinations. To be enriched over $\ConvexCat$, the composition operation must be bilinear (i.e., it must preserve convex combinations). Since composing joinings reduces to integrating Markov kernels, and integration is a linear operator, the composition distributes perfectly over convex combinations:
		\[ (t\lambda_{12} + (1-t)\lambda_{12}') \circ \lambda_{23} = t(\lambda_{12} \circ \lambda_{23}) + (1-t)(\lambda_{12}' \circ \lambda_{23}) .\]
		Thus, probability spaces as objects and joinings as morphisms form a category enriched over $\ConvexCat$.
		\item If the underlying measure spaces are standard Borel spaces (e.g., Polish spaces equipped with Borel measures), the set of joinings $\join(\mu, \nu)$ naturally carries the topology of weak convergence (or weak-* topology). By Prokhorov's Theorem, $\Pi(\mu, \nu)$ is a weakly compact, Hausdorff topological space.Furthermore, the Gluing Lemma composition is continuous with respect to this weak topology. Therefore, you can view this structure as a category enriched over $\CompHaus$ (Compact Hausdorff spaces). This topological enrichment is crucial when proving the existence of optimal transport plans via compactness arguments.
		\item By a similar reasoning, $\StochCat$ may be enriched over $\MetricCat$, the category of metric spaces and non-expansive maps.
	\end{enumerate}
	
	Enriching a category replaces the bare sets of morphisms (hom-sets) with objects from a structured background category, allowing the morphisms themselves to carry geometric, topological, or algebraic data. This shift fundamentally expands what categorical frameworks can compute and model. A Lawvere metric category, created by enrichment over the poset $([0, \infty], \geq, +)$is the exact categorical foundation of optimal transport and Wasserstein metrics. Enrichment over $\Topo$ or $\CompHaus$ equips the space of morphisms with a topology, thus opening the possibility of analyzing convergence and sequential limits within a categorical setting. Enriching over $\AbelCat$ (Abelian groups) or $\VectCat$ (vector spaces) yields additive and linear categories. This allows morphisms to be added, subtracted, and scaled—the strict prerequisite for defining kernels, cokernels, and chain complexes. 
	
	\section{Applications} \label{sec:applic}
	
	\paragraph{I. Finite multi-sets}  The category of finite sets equipped with multiset-valued morphisms—formally understood as the Kleisli category of the finite multiset monad—structures its morphisms as transition matrices populated by non-negative integer coefficients. In this categorical framework, these morphisms are naturally interpreted as \textit{possibilistic} maps with multiplicities, quantifying not merely whether a transition from one state to another is dynamically possible, but precisely how many distinct combinatorial pathways can realize that specific transition. This algebraic abstraction effectively bridges qualitative non-determinism with quantitative counting, providing a robust foundational tool for modeling resource-sensitive dynamical systems, branching probabilistic structures, and concurrent network flows.
	
	\begin{definition} [Finite Multi-sets]
		\cite{FritzKlinger2023d} The category $\mathrm{FinSetMult}$ is defined by
		\begin{itemize}
			\item Objects are finite sets
			\item Morphisms are set-valued maps
			\item Composition of $f: X \to Y$ and $g:Y \to Z$ is defined by
			\[
			g\circ f (x) := \bigcup_{y \in f(x)} g(y). 
			\]
		\end{itemize}
	\end{definition}
	
	This is a Markov category. Deterministic morphisms corresponds to single-valued maps.
	\begin{lemma}
		In $\mathrm{FinSetMult}$, a state $\mu: I \to X$ corresponds to a subset $A = \mu(*)$ of $X$. For a morphism $f:X \to X$, a state $\mu: I \to X$ is invariant if and only if $A = \mu(*) \subset X$ satisfies
		\[
		A = \bigcup_{x \in A} f(x).
		\]
	\end{lemma}
	We note that a morphism $f:X \to X$ corresponds to a directed graph with vertex set $X$. Invariant states define subgraphs.
	
	\begin{theorem}
		In $\mathrm{FinSetMult}$, a state $\mu: I \to X$ is ergodic with respect to a morphism $f:X \to X$ if and only if the corresponding graph is weakly connected.
	\end{theorem}
	\begin{proof}
		Let $\mu: I \to X$ be ergodic with respect to a morphism $f:X \to X$. We can define a deterministic morphism $c: X \to \mathbb{N}$ by assigning to which connected component $x \in X$ is contained. This is $f$-invariant by the construction. Thus, by ergodicity, $c\circ \mu$ is deterministic. This means $c$ takes the same value on $\mu(*)$. Therefore the graph with vertices $\mu(*)$ is weakly connected.
		
		Conversely, let the corresponding graph be weakly connected. If a deterministic morphism $c: X \to Y$ is $f$-invariant, we have $c(x) = c(y)$ for every edge $x \to y$ in the graph. By weak connectedness, $c$ assumes the same value on $\mu(*)$. This implies $c \circ \mu$ is deterministic.
	\end{proof}
	
	\begin{example}
		The following morphism $f:X \to X$ is ergodic on $X = \{1,2\}$:
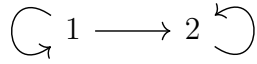
		\[
		\begin{tikzcd}
			1 \arrow[loop, distance=2em, in=215, out=145] \arrow[r] & 2 \arrow[loop, distance=2em, in=35, out=325]
		\end{tikzcd}
		\]
		Let $B= \{(1,1),(2,1), (2,2)\} \subset X \times X$. Then 
		$B$ is invariant under $f \times 1_X$, but it cannot be written as a direct product. Therefore, $f$ is ergodic but does not have Birkhoff property.
		
		Further, this observation shows that $f$ has more than two joinings with identity map $1_X$. 
	\end{example}
	\begin{lemma}
		Let $\mu: I \to X$ be an ergodic invariant state for a deterministic $f: X \to X$. Then, $\mu(*)$ consists of a  periodic orbit. 
	\end{lemma}
	\begin{proof}
		Let $A := \mu(*)$. By invariance and determinism, every $x \in A $ has $y \in A$ such that $x = f(y)$. Thus, $f|_A: A \to A$ is surjective. Since $A$ is finite, $f|_A$ is injective. Therefore $A$ consists of periodic orbits. By ergodicity, there can be at most one orbit.
	\end{proof}

	\begin{theorem}
		In $\mathrm{FinSetMult}$, deterministic ergodic systems have Birkhoff property.
	\end{theorem}
	
	\begin{proof}[Proof of Theorem]
		Let $f: X \to X$ be deterministic and $\mu: I \to X$ be ergodic invariant state. Let $\lambda: I \to X \times Y$ be an $f \times 1_Y$-invariant  state such that $X$ projection of $\lambda(*) =: B$ coincide with $\mu(*) = A$.
		
		For each $(x,y) \in B$, we have $A \times \{y\} \subset B$ by the invariance and the preceding lemma. Thus, $B$ is a direct product of the form $A\times C$.
	\end{proof}
	
	\paragraph{II. Multivalued maps} A multi-valued map \cite[e.g.]{BarmakEtAl2026conley}  is a function that assigns to every point in a finite domain, a closed subset of a target topological space. They are meant to emulate continuous flow-like behavior within a discrete setting, and the only requirement is that the maps be lower semi-continuous. This formalization is capable of transferring several notions and constructs of classical continuous dynamical systems into a discrete setting. By applying these maps to discrete structures like the face poset of a simplicial complex, infinite phase spaces can be effectively modeled as finite collections of cells governed by explicit topological transition rules. An even more useful property is that they extend the robust analytical machinery of Conley index theory \cite[e.g.]{Conley1978} to be directly applied to finite topological spaces. This provides a powerful computational framework to detect and isolate dynamic behaviors such as invariant sets and chaotic motion, directly from experimentally sampled point clouds and time-series.
	
	The Conley index of an isolated invariant set $S$ is an algebraic characterization of $S$ which does not require explicit knowledge of the specific isolated invariant set. Instead the characterization is based on an \textit{index pair}, an isolating neighborhood along with an exit-window for trajectories. If $\Time$ is continuous, then the Conley index can either be defined as the homotopy equivalence class of pointed topological space, or  as a homology module. 
	
	The main challenge in extending the Conley index to discrete time dynamics is the lack of a natural notion of homotopy along trajectories. Various new ideas have been used to bridge this gap, for example \textit{shape theory} by Robbin and Salamon \cite{RobbinSalamon1988}, \textit{cohomological Conley index} by Mrozek \cite{mrozek1990leray} and later generalized using the notion of \textit{shift equivalence} by Szymczak \cite{Szymczak1995conley} and others \cite{FranksRicheson2000shift}. Conley’s theory has successfully been extended to the case of multivalued discrete-time dynamics \cite{KaczynskiMrozek1995conley, BatkoMrozek2016weak, BatkoMrozek2017weak, batko2023morse, Stolot2006homotopy} but no work has yet been done for discrete space.
	

	\appendix
	\section{Markov categories} \label{sec:app:Markov}
	\subsection{Basic tools}
	
	In this section we recall some useful observations from the theory of Markov categories.
	
	\paragraph{1} Deletion of a branch after parallel composition is the same as an individual morphism. More precisely, the any two morphisms $f,g$ as shown on the left satisfy the condition shown on the right :
	\begin{equation} \label{eqn:kfad03}
		\forall
		\begin{tikzcd} A \arrow[d, "f"] \\ X \end{tikzcd} ,
		,
		\begin{tikzcd} A \arrow[d, "G"] \\ Y \end{tikzcd}
		\imply
		\begin{tikzcd}
			A \arrow[d, dashed, "f"] \arrow[rr, "\copyF_A^{(2)}"] && A \otimes A \arrow[d, "f\otimes g"] \\
			X && X\otimes Y \arrow[ll, "X\otimes \deleteF_Y"]
		\end{tikzcd}
	\end{equation}

	\paragraph{2} For any two morphisms $f,g : \Omega \to Y$ between the same pair of objects, equality of these morphisms imply almost sure equality w.r.t. any state $p:1\to \Omega$ on $\Omega$, which in turn implies that $f\circ p$ equals $g\circ p$ . In other words :
	\begin{equation}
		\forall
		\begin{tikzcd} 1 \arrow[d, "p"] \\ \Omega \end{tikzcd} ,\, 
		\begin{tikzcd} \Omega \arrow[d, "f"] \\ Y \end{tikzcd} ,\, 
		\begin{tikzcd} \Omega \arrow[d, "g"] \\ Y \end{tikzcd} \;:\;
		\begin{tikzcd} \Omega \arrow[d, "f"] \\ Y \end{tikzcd} = 
		\begin{tikzcd} \Omega \arrow[d, "g"] \\ Y \end{tikzcd} \;\Rightarrow\; 
		\begin{tikzcd} 1 \arrow[d, "p"] \\ \Omega \arrow[d, "\graph(f)"] \\ Y \end{tikzcd} = 
		\begin{tikzcd} 1 \arrow[d, "p"] \\ \Omega \arrow[d, "\graph{g}"] \\ Y \end{tikzcd} \;\Rightarrow\;   
		\begin{tikzcd} 1 \arrow[d, "p"] \\ \Omega \arrow[d, "f"] \\ Y \end{tikzcd} = 
		\begin{tikzcd} 1 \arrow[d, "p"] \\ \Omega \arrow[d, "g"] \\ Y \end{tikzcd} 
	\end{equation}

	\paragraph{3.} Disintegration of joint states : Consider the diagram on the left presenting a joint state $\lambda$ for two objects $X_1, X_2$ :
	\[\begin{tikzcd}
		& (X_1 \otimes X_2, \lambda_{1,2} ) \arrow[dl, "\pi_1"'] \arrow[dr, "\pi_2"] \\
		(X_1, \mu_1) & & (X_2, \mu_2)
	\end{tikzcd} \imply 
	\exists \begin{tikzcd} X_2 \arrow[d, "\left( \pi^{-1}_2 | \lambda \right)"] \\ X_1 \otimes X_2 \end{tikzcd} 
	\]
	The marginals of this joint state are respectively $\lambda_1, \lambda_2$. Such an arrangement means that we can construct the disintegration of $\lambda$ along the projection $\pi_2$, as indicated on the right. Applying the disintegration rules \eqref{eqn:def:disinteg:3} and \eqref{eqn:disinteg_invrsn} gives 
	\[ \graph^T(\pi_2^{-1}|\lambda) \circ \mu_2 = \graph(\pi_2) \circ \lambda .\]
	and
	\[ \pi_2\circ ( \pi_2^{-1}| \lambda ) = \Id_{X_2}, \quad \mu_2-\mbox{almost surely} . \]
	The first identity is an equality of two morphisms from $1$ to $X_1 \otimes X_2 \otimes X_2$. Thus their marginals over the first coordinate must be equal. So have
	\[\begin{split}
		\mathbb{E}( X_1 | X_2, \lambda ) \circ \mu_2
		&= \pi^{(X_1 \otimes X_2 \otimes X_2)}_{1} \circ \graph^T(\pi_2^{-1}|\lambda) \circ \mu_2 \\
		&= \pi^{(X_1 \otimes X_2 \otimes X_2)}_{1} \circ \graph(\pi_2) \circ \lambda \\
		&= \pi^{(X_1 \otimes X_2)}_{1} \circ \pi^{(X_1 \otimes X_2 \otimes X_2)}_{1,2} \circ \graph(\pi_2) \circ \lambda \\
		&= \pi^{(X_1 \otimes X_2)}_{1} \circ \Id_{X_1 \otimes X_2} \circ \lambda = \pi^{(X_1 \otimes X_2)}_{1} \circ \lambda \\
		&= \mu_1
	\end{split}\]
	In conclusion, we have proved : 
	\begin{equation} \label{eqn:fsoh0}
		\mathbb{E}( X_1 | X_2, \lambda ) \circ \mu_2 = \pi_1 \circ \paran{ \pi_2^{-1} | \lambda } \circ \mu_2 = \mu_1
	\end{equation}
	%
	
	\subsection{Bayesian inversion} 
	
	\begin{lemma} [Invariant and Bayesian inversion] \label{lem:inv_Bayes}
		\cite[Lem 5.2]{FritzGondaPerrone2021finetti} Consider two morphisms shown below on the left
		\begin{equation} \label{eqn:inv_Bayes}
			\forall \begin{tikzcd} 1 \arrow[d, "p"] \\ \Omega \arrow[d, "t"] \\ \Omega \end{tikzcd}
			\;:\;
			\begin{tikzcd} 1 \arrow[d, "p"'] \arrow[dr, "p"] \\ \Omega \arrow[r, "t"] & \Omega \end{tikzcd} 
			\quad:\quad
			\braces{\begin{array}{c} f \circ t^\dagger \\ = p -\mbox{a.s.} \\ f \end{array}}
			\imply
			\braces{\begin{array}{c} f \circ t \\ = p -\mbox{a.s.} \\ f \end{array}}
		\end{equation}
		such that the invariance condition shown in the center holds. Let $t^\dagger$ denote the Bayesian inverse of $t$. Then for every $p$-a.s. deterministic morphism $f:\Omega \to X$ the implication shown on the right above holds.
	\end{lemma}

	\[\begin{split}
		(f \otimes f) \circ \graph^T(t) \circ p &= (f \otimes f) \circ \graph(t^\dagger) \circ p , \quad \mbox{by \eqref{eqn:def:disinteg:3}}, \\
		&= (f \otimes f) \circ (\Id_{\Omega} \otimes f^{\dagger}) \circ \copyF_{\Omega} \circ p \\
		&= \SqBrack{ (f \circ \Id_{X}) \otimes (f\circ f^{\dagger}) } \circ \copyF_{\Omega} \circ p \\
		&= \SqBrack{ (f \circ \Id_{X}) \otimes f } \circ \copyF_{\Omega} \circ p , \quad \mbox{by assumptio}, \\
		&= (f\otimes f) \circ \copyF_{\Omega} \circ p.
	\end{split}\]

	\begin{lemma} \label{lem:fx93}
		The disintegration of any deterministic isomorphism is the inverse of that isomorphism.
	\end{lemma}

	\begin{equation} \label{eqn:oijd9}
		\forall \begin{tikzcd} 1 \arrow[d, "p"] \\ \Omega \end{tikzcd} ,\,
		\forall \begin{array}{c} \mbox{invetible}, \\ \mbox{deterministic} \end{array} 
		\begin{tikzcd} \Omega \arrow[d, "\psi"] \\ \Omega \end{tikzcd}
		\,:\,
		\begin{tikzcd} 1 \arrow[d, "p"] \arrow[dr, "p"] \\ \Omega \arrow[r, "\psi"] & \Omega \end{tikzcd}
		\imply
		\begin{tikzcd}
			&& \Omega \otimes \Omega \\
			\Omega \otimes \Omega \arrow[urr, "\cong"] \arrow[drr, "\cong"'] & 1 \arrow[l, dotted] \arrow[r, "p"] & \Omega \arrow[r, "\copyF_{\Omega}"] & \Omega \otimes \Omega \arrow[ul, "(f\circ \psi) \otimes \Omega"'] \arrow[dl, "f\otimes \psi^{-1}"] \\
			& & \Omega \otimes \Omega
		\end{tikzcd}
	\end{equation}
	To see why note that
	\[\begin{split}
		\SqBrack{ (f\circ \psi) \otimes \Omega } \circ \copyF_{\Omega} \circ p &= (f\otimes \Omega) \circ \graph^T(\psi) \circ p \\
		&= (f\otimes \Omega) \circ \graph(\psi^{-1}|p) \circ \psi \circ p , \quad \mbox{by \eqref{eqn:def:disinteg:3}}, \\
		&= (f\otimes \Omega) \circ \graph(\psi^{-1}|p) \circ p , \quad \mbox{by assumption}  \\
		&= (f\otimes \Omega) \circ \graph(\psi^{-1}) \circ p, \quad \mbox{by Lemma \ref{lem:fx93}} \\
		&= (f\otimes \psi^{-1}) \circ p .
	\end{split}\]
	%
	
	\subsection{Conditional independence} 
	
	In classical probability theory, two measurements on a probability space are considered to be conditional independent if their joint probability distribution is the product of their individual marginal distributions. This notion extends easily to categorical probability.
	
	\begin{definition} [Conditional independence]
		\cite[Def 2.10]{FritzRischel2020inf}
		Given a morphism $p : A \to X_1 \otimes \dots \otimes X_n$ , $p$ is said to display the conditional independence $X_1 \perp \dots \perp X_n \parallel A$ if the following equality holds
		\[ p = \otimes_{i=1}^{n} \paran{\pi_i \circ p} .\]
	\end{definition}
	
	Note that conditional independence does not depend on the order of the marginals. This is because if the conditional independence $X_1 \perp \dots \perp X_n \parallel A$ holds, then composing $p$ with any permutation $\sigma$ of the tensor factors $X_1, \dots, X_n$ gives a morphism $A \to X_{\sigma(1)} \otimes \dots \otimes X_{\sigma(n)}$ which displays the conditional independence $X_{\sigma(1)} \perp \dots \perp X_{\sigma(n)} \parallel A$. Therefore if $\bigotimes_{j \in F} X_j$ is a finite product without any particular order on the factors, then there is no ambiguity in writing $\perp_{j \in F} X_j \parallel A$.
	
	\begin{lemma} [Abstract Hewitt–Savage zero–one law] \label{lem:HewSav}
		\cite[Thm 5.4]{FritzRischel2020inf}
		Suppose that $\calC$ is a causal Markov category (Definition 2.7). Let $X_J$ be a Kolmogorov power of an object $X \in \calC$ with respect to an infinite set $J$. Suppose that morphisms $p : A \to X_J$ and deterministic $s : X_J \to T$ satisfy the following:$p$ displays the conditional independence $\perp_{i \in J} X_i \parallel A$.For every finite permutation $\sigma : J \to J$, we have $\hat{\sigma}p = p$ and $s\hat{\sigma} = s$.Then the composite $sp : A \to T$ is deterministic.
	\end{lemma}
	
	\subsection{Proof of Lemma \ref{lem:f4jd93}} \label{sec:proof:f4jd93}
	
	Lemma \ref{lem:f4jd93} is a direct consequence of Lemma \ref{lem:HewSav}. This is due to the fact that the Bernoulli measure is permutation invariant, and the projection morphisms are deterministic. \qed
	
	\section{Lemmas from ergodic theory} \label{sec:app:erg}
	\subsection{Proof of Lemma \ref{lem:graph_join:1}} \label{sec:proof:graph_join:1}
	
	We need to show that for every $t\in \Time$, the transformed state $\paran{ \iota \Phi_1^t \otimes \iota \Phi_2^t } \circ \paran{  \graph(h) \circ \mu_1 }$ is again $\graph(h) \circ \mu_1$. This identity is derived as follows :
	\[\begin{split}
		& \paran{ \iota \Phi_1^t \otimes \iota \Phi_2^t } \circ \paran{  \graph(h) \circ \mu_1 } = \paran{ \iota \Phi_1^t \otimes \iota \Phi_2^t } \circ \paran{ \iota_{\Omega_1} \otimes h } \circ \copyF_{\iota \Omega_1} \circ \mu_1 \\
		& \quad \quad = \SqBrack{ \paran{\iota \Phi_1^t \circ \iota_{\Omega_1}} \otimes \paran{\iota \Phi_2^t \circ h} } \circ \copyF_{\iota \Omega_1} \circ \mu_1 
		= \SqBrack{ \paran{\iota \Phi_1^t \circ \iota_{\Omega_1}} \otimes \paran{h\circ \iota \Phi_1^t} } \circ \copyF_{\iota \Omega_1} \circ \mu_1  \\
		& \quad \quad = \paran{ \iota \Phi_1^t \otimes h } \circ \paran{\iota_{\Omega_1} \otimes \Phi_1^t} \circ \copyF_{\iota \Omega_1} \circ \mu_1 = \paran{\iota_{\Omega_1} \otimes h } \circ \paran{ \iota \Phi_1^t \otimes \Phi_1^t} \circ \copyF_{\iota \Omega_1} \circ \mu_1 \\
		& \quad \quad = \paran{\iota_{\Omega_1} \otimes h } \circ \copyF_{\iota \Omega_1} \circ \iota \Phi_1^t \circ \mu_1 ,\; \mbox{by determinism of } \Phi_1 , \\
		& \quad \quad = \paran{\iota_{\Omega_1} \otimes h } \circ \copyF_{\iota \Omega_1} \circ \mu_1 ,\; \mbox{since } \mu_1 \mbox{ is invariant under } \Phi_1 , \\
		& \quad \quad = \graph(h) \circ \mu_1 .
	\end{split}\]
	This completes the proof of Lemma \ref{lem:graph_join:1}. \qed
	
	\subsection{Proof of Lemma \ref{lem:ptMarkov_Markv}} \label{sec:proof:ptMarkov_Markv}
	
	Since Lemma \ref{lem:pcm_cm} already establishes the monoidal properties of the pointed version, one only needs to verify the rules \eqref{eqn:def:MarkovCat:1}--\eqref{eqn:def:MarkovCat:3} of copy and deletion are  satisfied.

	\[\begin{tikzcd}
		1 \arrow[r, "x"] \arrow[rrr, "{\copyF_{X,x}}"', dotted, bend right] & X \arrow[rr, "\copyF_X"] &  & X\otimes X
	\end{tikzcd}\]
	\[\begin{tikzcd}
		&  &  & X \\
		1 \arrow[r, "x"] \arrow[rrr, "{\copyF_{X,x}}"', dotted, bend right] \arrow[rrru, "x", bend left] & X \arrow[rr, "\copyF_X"] \arrow[rru, "\cong"] &  & X\otimes X \arrow[u, "\pi_{i}"']
	\end{tikzcd}\]

	This completes the proof of Lemma \ref{lem:ptMarkov_Markv}. \qed
	\subsection{Proof of Theorem \ref{thm:StPrsrv}} \label{sec:proof:StPrsrv}
	
	We first show that (i) $\Rightarrow$ (ii). So we take a functor $\Phi : \Time \to \Comma{1}{\iota}$. The category $\Time$ has a unique object $\star_{\Time}$, and its image under $\Phi$ is a pointed object $1 \xrightarrow{\mu} \iota\Omega$ for some $\calX$-object $\Omega$. Note that for each morphism $t\in \Time$, the image under $\Phi$ is a morphism $\Phi^t : \Omega \to \Omega$. Its image $\iota \Phi^t : \iota \Omega \to \iota \Omega$ under $\iota$ must preserve $\mu$. Compositionality of the functor $\Phi$ implies
	\begin{equation}\label{eqn:thm:StPrsrv}
		\begin{tikzcd} [column sep = large]
			1 \arrow[d, "\mu"] \arrow[dr, "\mu"] \arrow[drr, bend left=10, "\mu"] \\
			\iota \Omega \arrow[r, "\Phi^{t}"] \arrow[rr, bend right=10, "\Phi^{t}"'] & \iota \Omega \arrow[r, "\Phi^{s}"] & \iota \Omega
		\end{tikzcd} , \quad \forall s,t\in \Time.
	\end{equation}
	The diagram indicates that $\Phi$ acts on the semi-group of endomorphisms of $\iota\Omega$ as well. With this interpretation, note that the state $\mu:1\to \iota\Omega$ acts as a natural transformation. This is precisely the pair described in (ii). The converse also follows from the same diagram. This completes the proof of Theorem \ref{thm:StPrsrv}. \qed
	\subsection{Proof of Lemma \ref{lem:subsig:1}} \label{sec:proof:subsig:1}
	
	This completes the proof of Lemma \ref{lem:subsig:1}. \qed
	\subsection{Proof of Lemma \ref{lem:sps_mrphsm_invert}} \label{sec:proof:sps_mrphsm_invert}
	
	We need to show that for every $t\in \Time$ :
	\[ \graph\paran{ (h^{-1}|\mu_X) \circ \iota \Phi^t_Y } \circ \mu_Y = \graph\paran{ \iota \Phi_X^t \circ (h^{-1}|\mu_X) } \circ \mu_Y . \]
	This can be derived as follows :
	\[\begin{split}
		& \graph\paran{ (h^{-1}|\mu_X) \circ \iota \Phi^t_Y } \circ \mu_Y = \SqBrack{ \paran{\iota\Omega_Y} \otimes \paran{ (h^{-1}|\mu_X) \circ \iota \Phi^t_Y } } \circ \copyF_Y \circ \mu_Y
	\end{split}\]
	Using \eqref{eqn:jsp3z} completes the proof of Lemma \ref{lem:sps_mrphsm_invert}. \qed 
	
	\section{Pointed categories} \label{sec:app:pointed}
	
	\subsection{Proof of Lemma \ref{lem:pcm_cm}} \label{sec:proof:pcm_cm}
	
	\begin{lemma} [Pointed versions of limits and colimits] \label{lem:jd9ps3}
		Let $\calC$ be a category with a terminal object $1$.
		\begin{enumerate} [(i)]
			\item If $\calC$ has all limits of pattern $J$, then so does its pointed version $\Comma{1}{\calC}$.
			\item If $\calC$ has all colimits of pattern $1\star J$, then $\Comma{1}{\calC}$ has all colimits of pattern $J$.
			\item If $\calC$ has all colimits of pattern $J$ for some connected category $J$, then so does $\Comma{1}{\calC}$.
		\end{enumerate}
	\end{lemma}
	\[\begin{tikzcd}
		& X \arrow[rd] \arrow[rrrd, "X \otimes (y\circ !_X)", bend left] &  &  &  \\
		1 \arrow[ru, "x"] \arrow[rd, "y"'] &  & X \wedge_{x,y} Y \arrow[rr, "\exists !", dotted] \arrow[dd, "!"'] &  & X\otimes Y \arrow[dd] \\
		& Y \arrow[ru] \arrow[rrru, "(x\circ !_Y) \otimes Y)", bend right] &  &  &  \\
		&  & 1 \arrow[rr] &  & (X,x) \otimes (Y,y)
	\end{tikzcd}\]
	Let $[X, Y]$ denote the internal hom in $\calC$. We construct the pointed internal hom $[(X, x), (Y, y)]_*$ via an equalizer in $\calC$.We have two canonical maps from $[X, Y]$ to $Y$:Evaluation at the basepoint: $[X, Y] \xrightarrow{x^*} [1, Y] \cong Y$The constant map at $y$: $[X, Y] \xrightarrow{!} 1 \xrightarrow{y} Y$Define the underlying object $[(X, x), (Y, y)]_*$ as the equalizer of these two maps.To point this object, note that the constant map $1 \to [X, Y]$ (the adjunct of $1 \times X \xrightarrow{!} 1 \xrightarrow{y} Y$) equalizes the two maps above. By the universal property of the equalizer, this constant map factors uniquely through a map $1 \to [(X, x), (Y, y)]_*$, establishing its basepoint.
	
	We must show a natural bijection of hom-sets in $[1, \calC]$:$$\hom_*(A \wedge X, Y) \cong \hom_*(A, [X, Y]_*)$$Let $f: A \wedge X \to Y$ be a pointed map. By the universal property of the smash product pushout, $f$ is uniquely determined by a map $\tilde{f}: A \times X \to Y$ in $\calC$ whose restriction to the wedge sum $A \vee Y$ factors through the terminal object to the basepoint $y: 1 \to Y$.This restriction imposes two conditions on $\tilde{f}$:$\tilde{f} \circ (id_A \times x) = y \circ !_A$ (It collapses $A \times \{x\}$ to $y$)$\tilde{f} \circ (a \times id_X) = y \circ !_X$ (It collapses $\{a\} \times X$ to $y$)Under the standard adjunction in $\calC$, $\tilde{f}: A \times X \to Y$ corresponds to a unique map $\hat{f}: A \to [X, Y]$. We translate the two conditions:Condition (1) means that for any generalized element of $A$, its image in $[X, Y]$ evaluates at $x$ to $y$. Categorically, $\hat{f}$ factors through the equalizer defining $[X, Y]_*$.Condition (2) means that the basepoint $a: 1 \to A$ is mapped by $\hat{f}$ to the constant map $1 \to [X, Y]$ at $y$. This means exactly that $\hat{f}: A \to [X, Y]_*$ preserves the basepoints.Thus, every pointed map out of the smash product corresponds to a unique pointed map into the internal hom, completing the proof that $[1, \calC]$ is a closed monoidal category.

	This completes the proof of Lemma \ref{lem:pcm_cm}. \qed
	\subsection{Proof of Lemma \ref{lem:jd9ps3}} \label{sec:proof:jd9ps3}
	
	\paragraph{Claim (i)} Assume $\calC$ has limits of shape $J$. Let $F: J \to \Comma{1}{\calC}$ be a diagram. We want to construct the limit of $F$ in $\Comma{1}{\calC}$. Applying the forgetful functor gives a diagram $\Forget \circ F: J \to \calC$. Since $\calC$ has limits of type $J$, this diagram has a limit $L := \lim(U \circ F)$ in $\calC$. This limit comes equipped with projection morphisms $\pi_j: L \to X_j$ for each $j \in Ob(J)$. The objects in the original diagram $F$ are of the form $(X_j, x_j)$. Note that the collection of base-points $x_j: 1 \to X_j$ naturally forms a cone over the diagram $U \circ F$ with vertex $1$. By the universal property of the limit $L$ in $\calC$, this cone induces a unique morphism $u: 1 \to L$ such that for all $j \in Ob(J)$ : $ \pi_j \circ u = x_j$. This unique map $u$ serves as the base-point for $L$, making $(L, u)$ an object in $\Comma{1}{\calC}$. Furthermore, the equation $\pi_j \circ u = x_j$ explicitly proves that the projections $\pi_j$ are valid morphisms in $\Comma{1}{\calC}$. To verify that $(L, u)$ is indeed the limit in $\Comma{1}{\calC}$, we must check its universal property. Let $(M, v)$ be another object in $\Comma{1}{\calC}$ equipped with a cone over $F$, consisting of projections $p_j: (M, v) \to (X_j, x_j)$. By definition, this means $p_j \circ v = x_j$ for all $j$. The $p_j$ maps form a cone over $U \circ F$ with vertex $M$ in $\calC$. The universal property of $L$ guarantees a unique morphism $h: M \to L$ in $\calC$ such that $\pi_j \circ h = p_j$. We now show that $h$ is a valid morphism in $\Comma{1}{\calC}$; namely, that $h \circ v = u$. We can prove this by post-composing with the limit projections $\pi_j$ : 
	\[\pi_j \circ (h \circ v) = (\pi_j \circ h) \circ v = p_j \circ v = x_j.\]
	This shows that the composition $(h \circ v)$ is a morphism $1 \to L$ that mediates the cone $x_j$. However, $u$ was defined as the unique morphism satisfying $\pi_j \circ u = x_j$. Therefore, it must be exactly true that $ h \circ v = u $.
	
	\paragraph{Claim (ii)} If $J$ is connected, for any two objects $j, k \in Ob(J)$, there is a zig-zag of morphisms connecting them. Because $F$ is a valid diagram in $\Comma{1}{\calC}$, every morphism $F(\alpha)$ preserves the base-point: $F(\alpha) \circ x_j = x_k$. The underlying colimit $C$ in $\calC$ with coprojections $\iota_j$ guarantees that:
	\[\iota_k \circ F(\alpha) = \iota_j.\]
	By substituting the base-point condition, we get : 
	\[\iota_k \circ x_k = \iota_k \circ (F(\alpha) \circ x_j) = (\iota_k \circ F(\alpha)) \circ x_j = \iota_j \circ x_j\]
	Because $J$ is connected, this equality cascades across the entire diagram. All the compositions $\iota_j \circ x_j$ are exactly the same morphism. This single, unique morphism serves as the canonical base-point $u: 1 \to C$.
	
	\paragraph{Claim (iii)} The proof follows similarly to Claim (ii). \qed

	\subsection{Proof of Lemma \ref{lem:disinteg_join}} \label{sec:proof:disinteg_join}
	
	First note that \eqref{eqn:def:disinteg:3} yields
	\begin{equation} \label{eqn:dfd6h}
		\graph^T(h_i^{-1}|\mu_i) \circ \mu_3 = \graph^T(h_i^{-1}|\mu_i) \circ h_i \circ \mu_i = \graph(h_i) \circ \mu_i, \quad i=1,2.
	\end{equation}
	\[\begin{split}
		& \SqBrack{ (h_1^{-1}|\mu_1) \otimes (h_2^{-1}|\mu_2) } \circ \copyF_{\iota \Omega_3} \circ \mu_3 = \SqBrack{ \iota \Omega_1 \otimes (h_2^{-1}|\mu_2) } \circ \SqBrack{ (h_1^{-1}|\mu_1) \otimes \iota \Omega_3 } \circ \copyF_{\iota \Omega_3} \circ \mu_3 \\
		& \quad \quad = \SqBrack{ \iota \Omega_1 \otimes (h_2^{-1}|\mu_2) } \circ \graph^T(h_1^{-1}|\mu_1) \circ \mu_3 = \SqBrack{ \iota \Omega_1 \otimes (h_2^{-1}|\mu_2) } \circ \graph(h_1) \circ \mu_1, \; \mbox{by \eqref{eqn:dfd6h}} \\
		& \quad \quad = \SqBrack{ \iota \Omega_1 \otimes (h_2^{-1}|\mu_2) } \circ \SqBrack{\iota \Omega_1 \otimes h_1} \circ \copyF_{\iota \Omega_1} \circ \mu_1
	\end{split}\]

	We need to show that
	\[\SqBrack{\iota \Phi_1^t \otimes \iota \Phi_2^t} \circ \join(\mu_1, \mu_2 \mid h_1, h_2) = \join(\mu_1, \mu_2 \mid h_1, h_2) .\]

	This can be derived as follows :
	\[\begin{split}
		& \SqBrack{\iota \Phi_1^t \otimes \iota \Phi_2^t} \circ \join(\mu_1, \mu_2 \mid h_1, h_2) = \SqBrack{\iota \Phi_1^t \otimes \iota \Phi_2^t} \circ \SqBrack{ (h_1^{-1}|\mu_1) \otimes (h_2^{-1}|\mu_2) } \circ \copyF_{\Omega_3} \circ \mu_3
	\end{split}\]

	This completes the proof of Lemma \ref{lem:disinteg_join}. \qed
	
	\subsection{Proof of Lemma \ref{lem:dij0s}} \label{sec:proof:dij0s}
	
	This completes the proof of Lemma \ref{lem:dij0s}. \qed
	\subsection{Proof of Lemma \ref{lem:vee_from_join}} \label{sec:proof:vee_from_join}
	
	We need to show that 
	\[ (i)\, \SqBrack{ \iota \Phi_3^t \otimes \iota \Phi_3^t } \circ \SqBrack{ (h_1 \otimes h_2) \circ \lambda} = (h_1 \otimes h_2) \circ \lambda; \quad (ii)\,  \pi_i \circ \SqBrack{ (h_1 \otimes h_2) \circ \lambda} = \mu_3, \, i=1,2. \]
	The identity (i) is derived as follows
	\[\begin{split}
		&\SqBrack{ \iota \Phi_3^t \otimes \iota \Phi_3^t } \circ \SqBrack{ (h_1 \otimes h_2) \circ \lambda} = \SqBrack{ \paran{\iota \Phi_3^t \circ h_1} \otimes \paran{\iota \Phi_3^t \circ h_2} } \circ \lambda = \SqBrack{ \paran{h_1 \circ \iota \Phi_1^t} \otimes \paran{h_1 \circ \iota \Phi_2^t} } \circ \lambda \\
		&\quad \quad = \paran{h_1 \otimes h_2} \circ \SqBrack{\iota \Phi_1^t \circ \iota \Phi_2^t} \circ \lambda = \paran{h_1 \otimes h_2} \circ \lambda.
	\end{split}\]
	To derive (ii) note that by \eqref{eqn:d4d5z}
	\[ \pi_i \circ \SqBrack{ (h_1 \otimes h_2) \circ \lambda} = \pi_i \circ \copyF \paran{ h_1 \circ \lambda, h_2 \circ \lambda } = h_i \circ \lambda = \mu_i. \]

	This leads to the diagram
	
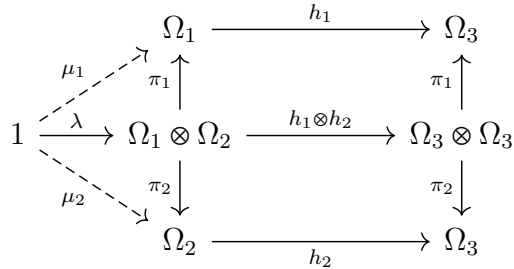
\[\begin{tikzcd}
		& \Omega_1 \arrow[rr, "h_1"] && \Omega_3 \\
		1 \arrow[dashed, dr, "\mu_2"'] \arrow[dashed, ur, "\mu_1"] \arrow[r, "\lambda"] & \Omega_1 \otimes \Omega_2 \arrow[rr, "h_1 \otimes h_2"] \arrow[u, "\pi_1"] \arrow[d, "\pi_2"'] && \Omega_3 \otimes \Omega_3 \arrow[u, "\pi_1"] \arrow[d, "\pi_2"'] \\
		& \Omega_2 \arrow[rr, "h_2"'] && \Omega_3
	\end{tikzcd}\]
	This diagram can be expanded using the dynamics $\Phi_3^t \otimes \Phi_3^t$ to get
	\[\begin{tikzcd}
		& && \Omega_1 \arrow[drr, "h_1", Holud] \\
		& \Omega_1 \arrow[urr, "\Phi_1^t"] \arrow[rr, "h_1"] && \Omega_3 \arrow[rr, "\Phi_3^t"] && \Omega_3\\
		1 \arrow[uurrr, bend left=40, Holud, "\mu_1"] \arrow[ddrrr, bend right=40, Akashi, "\mu_2"'] \arrow[dashed, dr, "\mu_2"'] \arrow[dashed, ur, "\mu_1"] \arrow[r, "\lambda"] & \Omega_1 \otimes \Omega_2 \arrow[rr, "h_1 \otimes h_2"] \arrow[u, "\pi_1"] \arrow[d, "\pi_2"'] && \Omega_3 \otimes \Omega_3 \arrow[u, "\pi_1"] \arrow[d, "\pi_2"'] \arrow[rr, "\Phi_3^t \otimes \Phi_3^t"] && \Omega_3 \otimes \Omega_3 \arrow[u, "\pi_1"] \arrow[d, "\pi_2"'] \\
		& \Omega_2 \arrow[drr, "\Phi_2^t"'] \arrow[rr, "h_2"'] && \Omega_3 \arrow[rr, "\Phi_3^t"'] && \Omega_3 \\
		& && \Omega_2 \arrow[urr, "h_1"', Akashi]
	\end{tikzcd}\]
	This completes the proof of Lemma \ref{lem:vee_from_join}. \qed
	\subsection{Proof of Lemma \ref{lem:a5ukdf}} \label{sec:a5ukdf}
	
	This completes the proof of Lemma \ref{lem:a5ukdf}. \qed
	
	\subsection{Proof of Lemma \ref{lem:dyn_colim}} \label{sec:proof:dyn_colim}
	
	Let $\eta : \psi \Rightarrow X$ be the colimiting cone. Fix any time $t\in\Time$. Because of the commutation assumed in the statement,
	\[ F(j)^t : \psi(j) \to \psi(j) , \quad \forall t\in \Time.\]
	\[\forall \begin{tikzcd}
		j \arrow[d, "h"'] \\ j'
	\end{tikzcd} \in J
	\;:\;
	\begin{tikzcd}
		\psi(j) \arrow[d, "\psi(h)"'] \arrow[rr, "F(j)^t"] && \psi(j) \arrow[d, "\psi(h)"]\\
		\psi(j') \arrow[rr, "F(j')^t"'] && \psi(j')
	\end{tikzcd}\]
	\[\forall \begin{tikzcd}
		j \arrow[d, "h"'] \\ j'
	\end{tikzcd} \in J
	\;:\;
	\begin{tikzcd}
		&& \psi(j) \arrow[ddll, "\eta_j"', Holud] \arrow[d, "\psi(h)"] \arrow[rr, "F(j)^t", Akashi] && \psi(j) \arrow[ddrr, "\eta_j", Akashi] \arrow[d, "\psi(h)"']\\
		&& \psi(j') \arrow[dll, "\eta_{j'}", Holud] \arrow[rr, "F(j')^t"', Akashi] && \psi(j') \arrow[drr, "\eta_{j'}"', Akashi] \\
		X \arrow[rrrrrr, dotted, Shobuj, "\exists ! \bar{\Phi}^t"] && && && X
	\end{tikzcd}\]

	This completes the proof of Lemma \ref{lem:dyn_colim}. \qed
	
	\subsection{Proof of identity \eqref{eqn:join_compose:1}} \label{sec:proof:join_compose:1}
	
	For every $i=1,3$ :
	\[\begin{split}
		\pi_i \circ \SqBrack{\lambda_{1,2} \circ \lambda_{2,3} } &= \pi_i \circ \SqBrack{ \mathbb{E}(X_1|X_2, \lambda_{1,2}) \otimes \mathbb{E}(X_3|X_2, \lambda_{2,3}) } \circ \copyF_{X_2}^{(2)} \circ \mu_2 = \mathbb{E}(X_i|X_2, \lambda_{2,i}) \circ \mu_2 \\
		&= \pi_i^{(i,2)} \circ \paran{ \pi_2^{(i,2)-1} | \mu_2 } \circ \mu_2 , \quad \mbox{by \eqref{eqn:kfad03}} , \\
		&= \mu_i , \quad \mbox{by \eqref{eqn:fsoh0}}.
	\end{split}\]
	This completes the proof of \eqref{eqn:join_compose:1}. \qed
	\section{Ergodic extensions} \label{sec:app:ext}
	\subsection{Proof of Lemma \ref{lem:0dj9xc}} \label{sec:proof:0dj9xc}

	This completes the proof of Lemma \ref{lem:0dj9xc}. \qed
	\subsection{Proof of Lemma \ref{lem:shift}} \label{sec:proof:shift}

	\[\begin{tikzcd}
		\otimes^{|J+t+s|} X & & \otimes^{|J+t|} X \arrow[ll, "+s"'] & & & & \otimes^{|J|} X \arrow[llll, "+t"'] \arrow[llllll, "+(t+s)", bend right] \\
		& {} \arrow[lu, "\pi_{J+t+s}"'] \arrow[ld, "\pi_{J'+t}"] & & \otimes^{|\Time|} X \arrow[ld, "\pi_{J'+t}"] \arrow[lu, "\pi_{J+t}"'] \arrow[ll, pos=0.2, "\shift^s"] & & \otimes^{|\Time|} X \arrow[ll, pos=0.7, "\shift^t"] \arrow[rd, "\pi_{J'}"'] \arrow[ru, "\pi_J"] \arrow[llll, pos=0.2, "\shift^{t+s}"', bend right=10]  \\
		\otimes^{|J'+t+s|} X \arrow[uu, pos=0.1, "\pi_{J'+t+s \to J+t+s}"] & & \otimes^{|J'+t|} X \arrow[uu, pos=0.1, "\pi_{J'+t \to J+t}"] \arrow[ll, "+s"] & & & & \otimes^{|J'|} X \arrow[uu, "\pi_{J'\to J}"'] \arrow[llll, "+t"] \arrow[llllll, "+(t+s)"', bend left] 
	\end{tikzcd}\]
	\[\begin{tikzcd}
		\otimes^{|J+t|} X & & & & \otimes^{|J|} X \\
		& \otimes^{|\Time|} X \arrow[ld, "\pi_{J'+t}"] \arrow[lu, "\pi_{J+t}"'] & & \otimes^{|\Time|} X \arrow[rd, "\pi_{J'}"'] \arrow[ru, "\pi_J"] \arrow[ll, "\shift^t"'] & \\
		\otimes^{|J'+t|} X \arrow[uu, "\pi_{J'+t\to J+t}"] & & 1 \arrow[rr, "\otimes^{|J'|} \mu"'] \arrow[ru, "\otimes^{|\Time|} \mu"] \arrow[rruu, "\otimes^{|J|} \mu", bend left=60] & & \otimes^{|J'|} X \arrow[uu, "\pi_{J'\to J}"']
	\end{tikzcd}\]
	\[\begin{tikzcd}
		\otimes^{|J+t|} X & & & & \otimes^{|J|} X \\
		& \otimes^{|\Time|} X \arrow[ld, "\pi_{J'+t}"] \arrow[lu, "\pi_{J+t}"'] & & \otimes^{|\Time|} X \arrow[rd, "\pi_{J'}"'] \arrow[ru, "\pi_J"] \arrow[ll, "\shift^t"'] & \\
		\otimes^{|J'+t|} X \arrow[uu, "\pi_{J'+t\to J+t}"] & & 1 \arrow[rr, "\otimes^{|J'|} \mu"'] \arrow[ru, "\otimes^{|\Time|} \mu"] \arrow[rruu, "\otimes^{|J|} \mu", bend left=60] \arrow[lu, "\otimes^{|\Time|} \mu"] \arrow[ll, "\otimes^{|J'+t|} \mu"] \arrow[lluu, "\otimes^{|J+t|} \mu"', bend right=49] & & \otimes^{|J'|} X \arrow[uu, "\pi_{J'\to J}"']
	\end{tikzcd}\]
	This completes the proof of Lemma \ref{lem:shift}. \qed
	
	\subsection{Proof of Theorem \ref{thm:fp9s}} \label{sec:proof:fp9s}
	
	Let $c:X \to Y$ be $f$-invariant and deterministic. Then, $c$ defines a morphism $c:(X,f, \mu) \to (Y, 1_Y, c \circ \mu)$. Definition \ref{def:graph_join} of graph joins and the diagram \eqref{eqn:def:graph_join:1} and the assumption imply
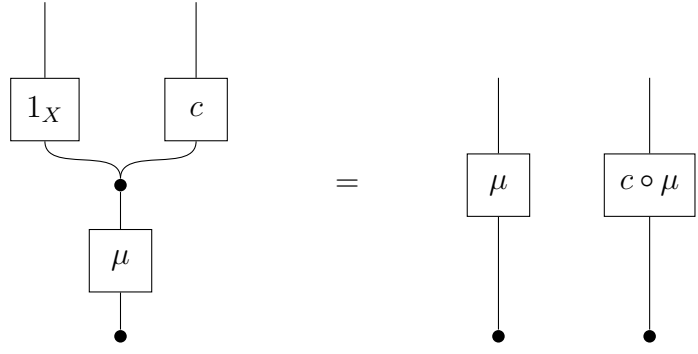
	\begin{center}
		\begin{tikzpicture}
			\node (e) at (5,-2) {=};
			\node[box=1/0/1/0] (n) at (1,-1) {1_X};
			\node[box=1/0/1/0] (h) at (3,-1) {c};
			\node[dot] (cpy) at (2,-2) {};
			\node[box=1/0/1/0] (mu) at (2,-3) {\mu};
			\node[box=1/0/1/0] (mu2) at (7,-2) {\mu};
			\node[box=1/0/1/0] (mu3) at (9,-2) {c \circ \mu};
			\node[dot] (i) at (2,-4) {};
			\node[dot] (i2) at (7,-4) {};
			\node[dot] (i3) at (9,-4) {};
			\wires{
				n = {south = cpy.north},
				h = {south = cpy.north},
				cpy = {south = mu.north},
				mu= {south = i.north},
				mu2 ={south = i2.north},
				mu3 ={south = i3.north},
			}{n.north, h.north, mu2.north, mu3.north}
		\end{tikzpicture}
	\end{center} 
	Therefore, we have
	
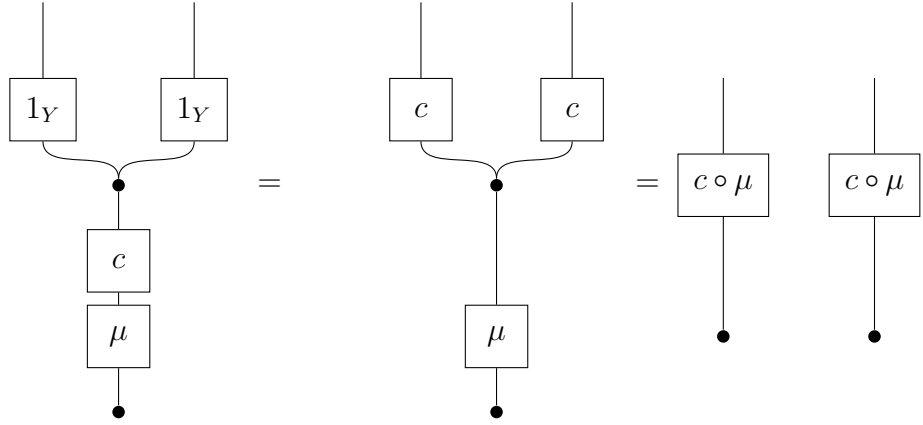
\begin{center}
		\begin{tikzpicture}
			\node (e) at (4,-2) {=};
			\node (e2) at (9,-2) {=};
			\node[box=1/0/1/0] (n) at (1,-1) {1_Y};
			\node[box=1/0/1/0] (h) at (3,-1) {1_Y};
			\node[box= 1/0/1/0] (f) at (2,-3) {c};
			\node[dot] (cpy) at (2,-2) {};
			\node[box=1/0/1/0] (mu) at (2,-4) {\mu};
			\node[dot] (i) at (2,-5) {};
			
			\node[box=1/0/1/0] (n2) at (6,-1) {c};
			\node[box=1/0/1/0] (h2) at (8,-1) {c};
			\node[dot] (cpy2) at (7,-2) {};
			\node[box=1/0/1/0] (mu2) at (7,-4) {\mu};
			\node[dot] (i2) at (7,-5) {};
			\node[box=1/0/1/0] (mu3) at (10,-2) {c \circ \mu};
			\node[box=1/0/1/0] (mu4) at (12,-2) {c \circ \mu};
			\node[dot] (i3) at (10,-4) {};
			\node[dot] (i4) at (12,-4) {};
			\wires{
				n = {south = cpy.north},
				h = {south = cpy.north},
				cpy = {south = f.north},
				f = {south =mu.north},
				mu= {south = i.north},
				n2 = {south = cpy2.north},
				h2 = {south = cpy2.north},
				cpy2 = {south = mu2.north},
				mu2= {south = i2.north},
				mu3 ={south = i3.north},
				mu4 ={south = i4.north}
			}{n.north, h.north,n2.north, h2.north, mu3.north, mu4.north}
		\end{tikzpicture}
	\end{center}
	Thus $c \circ \mu$ is deterministic.
	
	(2) Let $\lambda$ be a joining. Then, by the invariance and Birkhoff property, we have
	
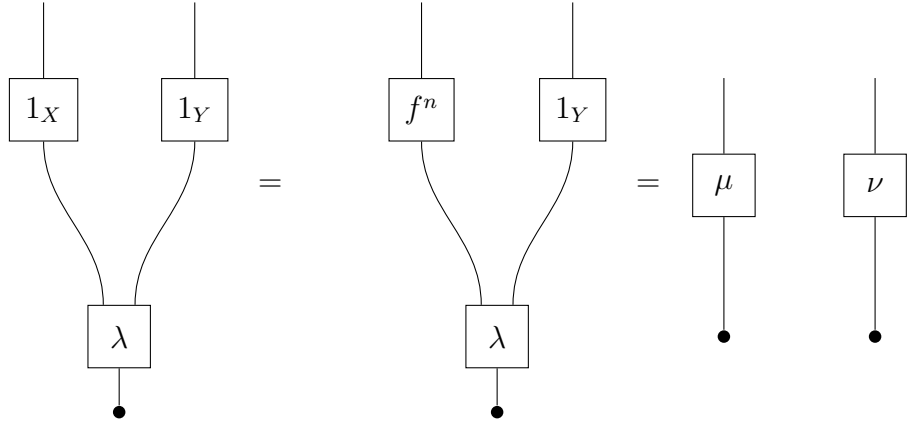
\begin{center}
		\begin{tikzpicture}
			\node (e) at (4,-2) {=};
			\node (e2) at (9,-2) {=};
			
			\node[box=1/0/1/0] (n) at (1,-1) {1_X};
			\node[box=1/0/1/0] (h) at (3,-1) {1_Y};
			\node[box=2/0/1/0] (mu) at (2,-4) {\lambda};
			\node[dot] (i) at (2,-5) {};
			
			\node[box=1/0/1/0] (n2) at (6,-1) {f^n};
			\node[box=1/0/1/0] (h2) at (8,-1) {1_Y};
			\node[box=2/0/1/0] (mu2) at (7,-4) {\lambda};
			\node[dot] (i2) at (7,-5) {};
			\node[box=1/0/1/0] (mu3) at (10,-2) {\mu};
			\node[box=1/0/1/0] (mu4) at (12,-2) {\nu};
			\node[dot] (i3) at (10,-4) {};
			\node[dot] (i4) at (12,-4) {};
			\wires{
				n = {south = mu.north.1},
				h = {south = mu.north.2},
				mu= {south = i.north},
				n2 = {south = mu2.north.1},
				h2 = {south = mu2.north.2},
				mu2= {south = i2.north},
				mu3 ={south = i3.north},
				mu4 ={south = i4.north}
			}{n.north, h.north,n2.north, h2.north, mu3.north, mu4.north}
		\end{tikzpicture}
	\end{center}
	
	This completes the proof of Theorem \ref{thm:fp9s}. \qed
	
	\subsection{Proof of Corollary \ref{cor:arith_cmplt}} \label{sec:proof:arith_cmplt}. 
	
	We begin with the observation that if an object $\mathbf{Y} \in \calS^{\text{join}, c}$ then for any non-trivial morphism $\phi  \mathbf{Y}'$, $\mathbf{Y}$ must also lie in $\calS^{\text{join}, c}$.
	
	We first prove that $\calS^{\text{join}, c} \subseteq \calS^\bot$. So suppose that $\mathbf{X}$ and $\mathbf{Y}$ are objects lying in $\calS$ and $\calS^{\text{join}, c}$ respectively. It has to be shown that $\mathbf{X}$ and $\mathbf{Y}$ are disjoint. The proof will be by contradiction. So suppose that there is a non-trivial join of these two objects. Then there is a non-trivial factor $\mathbf{Z}$ of $\mathbf{X}^\infty$ and $\mathbf{Y}$, as argued in Diagram \eqref{eqn:thm:push_pull} of Theorem \ref{thm:push_pull}. Then by assumption, $\mathbf{X}^\infty$ and $\mathbf{Y}$ lie in $\calC_1$ and $\calC_2$ respectively and therefore their factor $\mathbf{Z}$ lies in $\calC_1 \cap \calC_2$, which implies it must be $1$, leading to a contradiction. Thus $\mathbf{X}$ and $\mathbf{Y}$ are indeed disjoint.
	
	Next we prove that $\calS^\bot \subseteq \calS^{\text{join}, c}$. The proof will be by contradiction. So suppose that $\mathbf{Y}\in \calS^\bot$ with a common non-trivial factor with $\mathbf{X}$. Now Lemma \ref{lem:dij0s} applies and we can conclude that $\mathbf{X}$ and $\mathbf{Y}$ are not disjoint, a contradiction. \qed
\end{document}